\pdfoutput=1
\documentclass[bj]{imsart}
\usepackage[T1]{fontenc}      % Encoding
\usepackage[utf8]{inputenc}   % Encoding
\DeclareUnicodeCharacter{00A0}{ } % to prevent \u8 errors
\usepackage{microtype}        % Improved typesetting
\usepackage{bbm}              % For getting \mathbbm{1} since \mathbb{1} does not work properly
\usepackage{amsmath}
\usepackage{amsthm}
\usepackage{amssymb}
\usepackage{commath}          % Commands for differentials etc.
\usepackage{mathtools}
\usepackage{thmtools}
\usepackage[numbers]{natbib}
\usepackage{verbatim}               % For multiline comments
\usepackage[dvipsnames]{xcolor}     % For writer-specific comments, must be loaded before tikz

\usepackage[colorlinks,citecolor=blue,urlcolor=blue]{hyperref}

\allowdisplaybreaks

\numberwithin{equation}{section}
\declaretheorem[Refname={Theorem,Theorems}]{theorem}
\numberwithin{theorem}{section} % This should commented out if there is to be a non-section-based nvumbering
\declaretheorem[style=definition,numberlike=theorem,Refname={Definition,Definitions}]{definition}

\declaretheorem[style=definition,numberlike=theorem,Refname={Remark,Remarks}]{remark}
\declaretheorem[numberlike=theorem,Refname={Lemma,Lemmas}]{lemma}
\declaretheorem[name=Corollary,numberlike=theorem,Refname={Corollary,Corollaries}]{corollary}
\declaretheorem[name=Proposition,numberlike=theorem,Refname={Proposition,Propositions}]{proposition}

\newtheoremstyle{named}{}{}{\itshape}{}{\bfseries}{.}{.5em}{\thmnote{#3}}
\theoremstyle{named}
\newtheorem*{namedtheorem}{Theorem}

\DeclareMathOperator{\e}{e} % Exponential constant

\DeclarePairedDelimiterX\Set[2]{\lbrace}{\rbrace}%
{ #1 \,:\, #2 }                                         % Set of type {A : B}
\DeclarePairedDelimiterX\inprod[2]{\langle}{\rangle}%
{ #1 , #2 }                                             % Inner product
\newcommand{\R}{\mathbb{R}} % Set of real numbers
\newcommand{\N}{\mathbb{N}} % Set of natural numbers
\newcommand{\GP}{\mathcal{GP}}
\DeclareMathOperator{\tr}{tr}

\makeatletter
\def\blfootnote{\gdef\@thefnmark{}\@footnotetext}
\makeatother

\begin{document}

\begin{frontmatter}
  \title{Small Sample Spaces for Gaussian Processes}
  \runauthor{Toni Karvonen}
  \runtitle{Small Sample Spaces for Gaussian Processes}

\begin{aug}
  \author[A,B]{\fnms{Toni} \snm{Karvonen}\ead[label=e1,mark]{toni.karvonen@helsinki.fi}}
  \address[A]{The Alan Turing Institute, 96 Euston Road, London NW1 2DB, United Kingdom}
    \address[B]{University of Helsinki, Department of Mathematics and Statistics, Pietari Kalmin katu 5, 00560 Helsinki, Finland \\
  \printead{e1}}
\end{aug}

\begin{abstract}
It is known that the membership in a given reproducing kernel Hilbert space (RKHS) of the samples of a Gaussian process $X$ is controlled by a certain nuclear dominance condition. However, it is less clear how to identify a ``small'' set of functions (not necessarily a vector space) that contains the samples. This article presents a general approach for identifying such sets. We use \emph{scaled RKHSs}, which can be viewed as a generalisation of Hilbert scales, to define the \emph{sample support set} as the largest set which is contained in every element of full measure under the law of $X$ in the $\sigma$-algebra induced by the collection of scaled RKHS. This potentially non-measurable set is then shown to consist of those functions that can be expanded in terms of an orthonormal basis of the RKHS of the covariance kernel of $X$ and have their squared basis coefficients bounded away from zero and infinity, a result suggested by the Karhunen--Lo\`{e}ve theorem.
\end{abstract}


\end{frontmatter}

\section{Introduction}

\noindent
Let $K \colon T \times T \to \R$ be a positive-semidefinite kernel on a set $T$ and consider any Gaussian process $(X(t))_{t \in T}$ with mean zero and covariance $K$, which we denote $(X(t))_{t \in T} \sim \GP(0, K)$.
Let $H(K)$ be the reproducing kernel Hilbert space (RKHS) of $K$ equipped with inner product $\inprod{\cdot}{\cdot}_K$ and norm $\norm[0]{\cdot}_K$.
It is a well-known fact, apparently originating with \citet{Parzen1963}, that the samples of $X$ are not contained in $H(K)$ if this space is infinite-dimensional. 
Furthermore, \citet{Driscoll1973}; \citet{Fortet1973}; and \citet{LukicBeder2001} have used the zero-one law of \citet{Kallianpur1970} to show essentially that, given another kernel $R$ and under certain mild assumptions,
\begin{equation*}
  \mathbb{P}\big[ X \in H(R) \big] = 1 \: \text{ if $R \gg K$} \quad \text{ and } \quad \mathbb{P}\big[ X \in H(R) \big] = 0 \: \text{ if $R \not\gg K$},
\end{equation*}
where $R \gg K$ signifies that $R$ \emph{dominates} $K$ (i.e., $H(K) \subset H(R)$) and, moreover, that the dominance is \emph{nuclear} (see Section~\ref{sec:domination} for details). 
This \emph{Driscoll's theorem} is an exhaustive tool for verifying whether or not the samples from a Gaussian process are contained in a given RKHS.
A review of the topic can be found in~\citep[Chapter~4]{Gualtierotti2015}.
Two questions now arise:
\begin{itemize}
\item How to construct a kernel $R$ such that $R \gg K$?
\item Is it possible to exploit the fact that $\mathbb{P}[X \in H(R_1) \setminus H(R_2)] = 1$ for any kernels such that $R_1 \gg K$ and $R_2 \not\gg K$ to identify in some sense the smallest set of functions which contains the samples with probability one?
\end{itemize}
Answers to questions such as these are instructive for theory and design of Gaussian process based learning~\citep{Flaxman2016, VaarZanten2011}, emulation and approximation~\citep{Karvonen-MLE2020, XuStein2017}, and optimisation~\citep{Bull2011} methods.
For simplicity we assume that the domain $T$ is a complete separable metric space, that the kernel $K$ is continuous and its RKHS is separable, and that the samples of $X$ are continuous.
Although occasionally termed ``rather restrictive''~\citep[p.\@~309]{Driscoll1973}, these continuity assumptions are satisfied by the vast majority of domains and Gaussian processes commonly used in statistics and machine learning literature~\citep{RasmussenWilliams2006, Stein1999}, such as stationary processes with Gaussian or Mat\'{e}rn covariance kernels.
Our motivation for imposing these restrictions is that they imply that RKHSs are measurable.

\subsection{Contributions}

First, we present a flexible construction for a kernel $R$ such that $R \gg K$.
For any orthonormal basis $\Phi = (\phi_n)_{n=1}^\infty$ of $H(K)$ the kernel $K$ can be written as $K(t, t') = \sum_{n=1}^\infty \phi_n(t) \phi_n(t')$ for all $t, t' \in T$.
Given a positive sequence $A = (\alpha_n)_{n=1}^\infty$ such that $\sum_{n=1}^\infty \alpha_n \phi_n(t)^2 < \infty$ for all $t \in T$, we define the \emph{scaled kernel}
\begin{equation} \label{eq:scaled-kernel-intro}
  K_{A,\Phi}(t, t') = \sum_{n=1}^\infty \alpha_n \phi_n(t) \phi_n(t').
\end{equation}
This is a significant generalisation of the concept of powers of kernels which has been previously used to construct RKHSs which contain the samples by \citet{Steinwart2019}.
We call the sequence $A$ a $\Phi$-\emph{scaling} of $H(K)$.
If $\alpha_n \to \infty$ as $n \to \infty$, the corresponding \emph{scaled RKHS}\footnote{These spaces are not to be confused with classical Hilbert scales defined via powers of a strictly positive self-adjoint operator; see~\citep[Section~8.4]{EnglHankeNeubauer1996} and~\citep{KreinPetunin1966}.}, $H(K_{A,\Phi})$, is a proper superset of $H(K)$, though not necessarily large enough to contain the samples of $X$.
We show that convergence of the series $\sum_{n=1}^\infty \alpha_n^{-1}$ controls whether or not samples are contained in $H(K_{A,\Phi})$.
If $\sum_{n=1}^\infty \alpha_n^{-1}$ converges slowly, we can therefore interpret $H(K_{A,\Phi})$ as a ``small'' RKHS which contains the samples.

\begin{namedtheorem}[Main Result I (Theorem~\ref{thm:driscoll-scaled})]
  Let $\Phi = (\phi_n)_{n=1}^\infty$ be an orthonormal basis of $H(K)$ and $A = (\alpha_n)_{n=1}^\infty$ a $\Phi$-scaling of $H(K)$. If $K_{A,\Phi}$ is continuous and $d_K(t,t') = \norm[0]{K(\cdot,t) - K(\cdot, t')}_K$ is a metric on $T$, then either
  \begin{equation*}
    \mathbb{P}\big[ X \in H(K_{A,\Phi}) \big] = 0 \:\: \text{ and } \:\: \sum_{n=1}^\infty \frac{1}{\alpha_n} = \infty \quad \text{ or } \quad \mathbb{P}\big[ X \in H(K_{A,\Phi}) \big] = 1 \:\: \text{ and } \:\: \sum_{n=1}^\infty \frac{1}{\alpha_n} < \infty.
  \end{equation*}
\end{namedtheorem}

In Section~\ref{sec:examples}, we use this result to study sample path properties of Gaussian processes defined by infinitely smooth kernels.
These appear to be the first sufficiently descriptive results of their kind.
An example of an infinitely smooth kernel that we consider is the univariate (i.e., $T \subset \R$) Gaussian kernel \sloppy{${K(t,t') = \exp(-(t-t')^2/(2\ell^2))}$} with length-scale $\ell > 0$ for which we explicitly construct several scaled kernels $R$ whose RKHSs are ``small'' but still contain the samples of $(X(t))_{t \in T} \sim \GP(0, K)$.
In Section~\ref{sec:MLE}, Theorem~\ref{thm:driscoll-scaled} is applied to provide an intuitive explanation for a conjecture by \citet{XuStein2017} on asymptotic behaviour of the maximum likelihood estimate of the scaling parameter of the Gaussian kernel when the data are generated by a monomial function on a uniform grid.

Secondly, we use Theorem~\ref{thm:driscoll-scaled} to construct a ``small'' set which ``almost'' contains the samples. 
This \emph{sample support set} is distinct from the traditional topological support of a Gaussian measure; see the discussion at the end of Section~\ref{sec:related}.
Let $C(T)$ denote the set of continuous function on $T$.

\begin{namedtheorem}[Main Result II (Theorems~\ref{thm:sample-set-all} and~\ref{thm:not-countable})] \label{main-result-ii}
  Let $\Phi = (\phi_n)_{n=1}^\infty$ be an orthonormal basis of $H(K)$ and suppose there is a $\Phi$-scaling $A = (\alpha_n)_{n=1}^\infty$ such that $\sum_{n=1}^\infty \alpha_n^{-1} < \infty$ and $K_{A,\Phi}$ is continuous.
  Let $\mathcal{S}(\mathfrak{R})$ be the $\sigma$-algebra generated by the collection of scaled RKHSs consisting of continuous functions and $S_\mathfrak{R}(K)$ the largest subset of $C(T)$ that is contained in every $H \in \mathcal{S}(\mathfrak{R})$ such that $\mathbb{P}[X \in H] = 1$.
  Suppose that $d_K(t,t') = \norm[0]{K(\cdot,t) - K(\cdot, t')}_K$ is a metric on $T$.
  Then $S_\mathfrak{R}(K)$ consists precisely of the functions $f = \sum_{n=1}^\infty f_n \phi_n$ such that
  \begin{equation} \label{eq:constant-order-intro}
    \liminf_{n \to \infty} f_n^2 > 0 \quad \text{ and } \quad \sup_{n \geq 1} f_n^2 < \infty.
  \end{equation}
  Furthermore, for every $H \in \mathcal{S}(\mathfrak{R})$ such that $\mathbb{P}[X \in H] = 1$ there exists $F \in \mathcal{S}(\mathfrak{R})$ such that $S_\mathfrak{R}(K)$ is a proper subset of $F$ and $F$ is a proper subset of $H$.
\end{namedtheorem}

The set $S_{\mathfrak{R}}(K)$ may fail to be measurable.
The latter part of the above result is therefore important in demonstrating that it is possible to construct sets which are arbitrarily close to $S_{\mathfrak{R}}(K)$ and contain the samples using countably many elementary set operations of scaled RKHSs.
At its core this is a manifestation of the classical result that there is no meaningful notion of a boundary between convergent and divergent series~\citep[§~41]{Knopp1951}.
The general form of the Karhunen--Lo\`{e}ve theorem is useful in explaining the characterisation in~\eqref{eq:constant-order-intro}.
If $(\phi_n)_{n=1}^\infty$ is any orthonormal basis of $H(K)$, then the Gaussian process can be written as $X(t) = \sum_{n=1}^\infty \zeta_n \phi_n(t)$ for all $t \in T$, where $\zeta_n$ are independent standard normal random variables.
The series converges in $L^2(\mathbb{P})$, but if almost all samples of $X$ are continuous, convergence is also uniform on $T$ with probability one~\citep[Theorem~3.8]{Adler1990}.
Because $\norm[0]{X(t)}_K^2 = \sum_{n=1}^\infty \zeta_n^2$ and $\mathbb{E}[\zeta_n^2] = 1$ for every $n$, the Karhunen--Lo\`{e}ve expansion suggests, somewhat informally, that the samples are functions $f = \sum_{n=1}^\infty f_n \phi_n$ for which the sequence $(f_n^2)_{n=1}^\infty$ satisfies~\eqref{eq:constant-order-intro}.

\subsection{On Measurability and Continuity} \label{sec:measurability}

Suppose for a moment that $T$ is an arbitrary uncountable set, $K$ a positive-semidefinite kernel on $T$ such that its RKHS $H(K)$ is infinite-dimensional, and $(X(t))_{t \in T} \sim \mathcal{GP}(0, K)$ a generic Gaussian process defined on a generic probability space $(\Omega, \mathcal{A}, \mathbb{P})$.
Let $\R^T$ be the collection of real-valued functions on $T$ and $\tilde{\mathcal{B}}$ the $\sigma$-algebra generated by cylinder sets of the form $\Set{f \in \R^T}{ (f(t_1), \ldots, f(t_n)) \in B^n}$ for any $n \in \N$ and any Borel set $B^n \subset \R^n$.
Let $\Phi_X(\omega) = X(\cdot, \omega)$. Then $\tilde{\mu}_X = \mathbb{P} \circ \Phi_X^{-1}$ is the law of $X$ on the measurable space $(\R^T, \tilde{\mathcal{B}})$.
Consequently,
  $\mathbb{P}[ X \in H ] = \tilde{\mu}_X\big( \Set{ \omega \in \Omega }{ X(\cdot, \omega) \in H } \big)$
for $H \in \tilde{\mathcal{B}}$.
Let $(\R^T, \tilde{\mathcal{B}}_0, \tilde{\mu}_{X,0})$ be the completion of $(\R^T, \tilde{\mathcal{B}}, \tilde{\mu}_X)$ and $R \colon T \times T \to \R$ a positive-semidefinite kernel and $H(R)$ its RKHS.
The following facts are known about the measurability of $H(K)$ and $H(R)$:
\begin{itemize}
\item In general, $H(R) \notin \tilde{\mathcal{B}}$. For example, if $T$ is equipped with a topology and $R$ is continuous on $T \times T$, then $H(R) \subset C(T)$. However, no non-empty subset of $C(T)$ can be an element of $\tilde{\mathcal{B}}$.
\item \citet[p.\@~347]{LePage1973} has proved that $H(K) \in \tilde{\mathcal{B}}_0$ and $\tilde{\mu}_{X,0}(H(K)) = 0$, a claim which originates with \citet{Parzen1959,Parzen1963}.
  A version which requires separability and continuity appears in~\citep{Kallianpur1971}.
\item If the RKHS $H(R)$ is infinite-dimensional and $R \not\gg K$, then $H(R) \in \tilde{\mathcal{B}}_0$ and \sloppy{${\tilde{\mu}_{X,0}(H(R)) = 0}$}. This result is contained in the proof of Theorem~7.3 in~\citep{LukicBeder2001}. See also~\citep[Proposition~4.5.1]{Gualtierotti2015}.
\item \citet[Corollary~2]{LePage1973} has proved a dichotomy result which states that if $G \subset \R^T$ is an additive group  and $G \in \tilde{\mathcal{B}}_0$, then either $\tilde{\mu}_{X,0}(G) = 0$ or $\tilde{\mu}_{X,0}(G) = 1$. Furthermore, $\tilde{\mu}_{X,0}(G) = 1$ implies that $H(K) \subset G$. This is a general version of the zero-one law of \citet[Theorem~2]{Kallianpur1970}.
\end{itemize}
It appears that not much more can be said without imposing additional structure or constructing versions of $X$, as is done in~\citep{LukicBeder2001}.

Suppose that $(T, d_T)$ is a complete separable metric space, that the kernel $K$ is continuous, and that almost all samples of $(X(t))_{t \in T} \sim \mathcal{GP}(0, K)$ are continuous, which is to say that the $(\tilde{\mathcal{B}}_0, \tilde{\mu}_{X,0})$-outer measure of $C(T)$ is one.
Define the probability space $(C(T), \mathcal{B}, \mu_X)$ as
\begin{equation} \label{eq:B-sigma-algebra}
  \mathcal{B} = C(T) \cap \tilde{\mathcal{B}}_0 \quad \text{ and } \quad \mu_X(C(T) \cap H) = \tilde{\mu}_{X,0}(H) \: \text{ for } \: H \in \tilde{\mathcal{B}}_0.
\end{equation}
The rest of this article is concerned with $(C(T), \mathcal{B}, \mu_X)$ and it is to be understood that $\mathbb{P}[ X \in H ]$ stands for $\mu_X(H)$ for any $H \in \mathcal{B}$.
In this setting \citet[p.\@~313]{Driscoll1973} has proved that $H(R) \in \mathcal{B}$ if $R$ is continuous and positive-definite.
By using Theorem~1.1 in~\citep{Fortet1973} (Theorem~4.1 in~\citep{LukicBeder2001}) one can generalise this result for a continuous and positive-semidefinite $R$; see the proof of Theorem~7.3 in~\citep{LukicBeder2001}.

\subsection{Notation and Terminology}

For non-negative real sequences $(a_n)_{n=1}^\infty$ and $(b_n)_{n=1}^\infty$ we write $a_n \preceq b_n$ if there is $C > 0$ such that $a_n \leq C b_n$ for all sufficiently large $n$. If both $a_n \preceq b_n$ and $b_n \preceq a_n$ hold, we write $a_n \asymp b_n$. If $a_n / b_n \to 1$ as $n \to \infty$, we write $a_n \sim b_n$.
For two sets $F$ and $G$ we use $F \subsetneq G$ to indicate that $F$ is a proper subset of $G$.
A kernel $R \colon T \times T \to \R$ is \emph{positive-semidefinite} if
  \begin{equation} \label{eq:psd}
    \sum_{i=1}^N \sum_{j=1}^N a_i a_j R(t_i, t_j) \geq 0
  \end{equation}
  for any $N \geq 1$, $a_1, \ldots, a_N \in \R$ and $t_1, \ldots, t_N \in T$.
  In the remainder of this article positive-semidefinite kernels are simply referred to as \emph{kernels}.
  If the inequality in~\eqref{eq:psd} is strict for any pairwise distinct $t_1, \ldots, t_N$, the kernel is said to be \emph{positive-definite}.

\subsection{Standing Assumptions}

For ease of reference our standing assumptions are collected here.
We assume that (i) $(T, d_T)$ is a \emph{complete separable metric space}; (ii) the covariance kernel $K \colon T \times T \to \R$ is \emph{continuous} and \emph{positive-semidefinite} on $T \times T$; (iii) the RKHS $H(K)$ induced by $K$ is \emph{infinite-dimensional} and \emph{separable}\footnote{Most famously, separable RKHSs are induced by Mercer kernels, which are continuous kernels defined on compact subsets of $\R^d$~\citep[Section~11.3]{Paulsen2016}.}; and (iv) $(X(t))_{t \in T} \sim \GP(0, K)$ is a zero-mean Gaussian process on a probability space $(\Omega, \mathcal{A}, \mathbb{P})$ with \emph{continuous paths}.
The law $\mu_X$ of $X$ is defined on the measurable space $(C(T), \mathcal{B})$ which was constructed in Section~\ref{sec:measurability}.
Some of our results have natural generalisations for general second-order stochastic processes; see~\citep{LukicBeder2001}, in particular Sections~2 and~5, and~\citep{Steinwart2019}.
We do not pursue these generalisations.

\section{Related Work} \label{sec:related}

Reproducing kernel Hilbert spaces which contain the samples of $(X(t))_{t \in T} \sim \mathcal{GP}(0, K)$ have been constructed by means of integrated kernels in~\citep{Lukic2004}, convolution kernels in~\citep{Cialenco2012} and~\citep[Section~3.1]{Flaxman2016}, and, most importantly, \emph{powers of RKHSs}~\citep{SteinwartScovel2012} in~~\citep[Section~4]{Kanagawa2018} and~\citep{Steinwart2019}.
Namely, let $T$ be a compact metric space, $K$ a continuous kernel on $T \times T$, and $\nu$ a finite and strictly positive Borel measure on $T$.
Then the integral operator $T_\nu$, defined for $f \in L^2(\nu)$ via
\begin{equation} \label{eq:mercer-operator}
  (T_\nu f)(t) = \int_T K(t, t') f(t') \dif \nu(t'),
\end{equation}
has decreasing and positive eigenvalues $(\lambda_n)_{n=1}^\infty$, which vanish as $n \to \infty$, and eigenfunctions $(\psi_n)_{n=1}^\infty$ in $H(K)$ such that $(\sqrt{\lambda_n} \psi_n)_{n=1}^\infty$ is an orthonormal basis of $H(K)$. The kernel has the uniformly convergent Mercer expansion $K(t, t') = \sum_{n=1}^\infty \lambda_n \psi_n(t) \psi_n(t')$ for all $t, t' \in T$.
For $\theta > 0$, the $\theta$th power of $K$ is defined as
\begin{equation} \label{eq:power-kernel-introduction}
  K^{(\theta)}(t, t') = \sum_{n=1}^\infty \lambda_n^\theta \psi_n(t) \psi_n(t').
\end{equation}
The series~\eqref{eq:power-kernel-introduction} converges if $\sum_{n=1}^\infty \lambda_n^\theta \psi_n(t)^2 < \infty$ for all $t \in T$.
Furthermore, $H(K^{(\theta_2)}) \subsetneq H(K^{(\theta_1)})$ if $\theta_1 < \theta_2$ and $\mathbb{P}[ X \in H(K^{(\theta)})] = 1$ if and only if $\sum_{n=1}^\infty \lambda_n^{1-\theta} < \infty$~\citep[Theorem~5.2]{Steinwart2019}.
When it comes to sample properties, the power kernel construction has two significant downsides:
(i) The measure $\nu$ is a \emph{nuisance parameter}. If one is only interested in sample path properties of Gaussian processes this measure should not have an intrinsic part to play in the analysis and results.
(ii) The construction is somewhat \emph{inflexible} and \emph{unsuitable for infinitely smooth kernels}. Because $H(K^{(\theta)})$ consists precisely of the functions
$f = \sum_{n=1}^\infty f_n \lambda_n^{1/2} \psi_n$ such that $\sum_{n=1}^\infty f_n^2 \lambda_n^{1-\theta} < \infty$
and $\lambda_n \to 0$ as $n \to \infty$, how much larger $H(K^{(\theta)})$ is than $H(K) = H(K^{(\theta=1)})$ is determined by rate of decay of the eigenvalues.
Power RKHSs are more descriptive and fine-grained when the kernel is finitely smooth and its eigenvalues have polynomial decay $n^{-a}$ for $a > 0$ (e.g., Matérn kernels) than when the kernel is infinitely smooth with at least exponential eigenvalue decay $\e^{-bn}$ for $b > 0$ (e.g., Gaussian): the change the decay condition $\sum_{n=1}^\infty f_n^2 < \infty$ for the coefficients $(f_n)_{n=1}^\infty$ to $\sum_{n=1}^\infty f_n^2 n^{-a(1-\theta)} < \infty$ is arguably less substantial than that from $\sum_{n=1}^\infty f_n^2 < \infty$ to $\sum_{n=1}^\infty f_n^2 \e^{-b(1-\theta)n} < \infty$.
Indeed, as pointed out by \citet[Section~4.4]{Kanagawa2018}, when the kernel is Gaussian \emph{every} $H(K^{(\theta)})$ with $\theta < 1$ contains the samples with probability one, which renders powers of RKHSs of dubious utility in that setting because $H(K^{(\theta=1)}) = H(K)$ does not contain the samples.
The relationship between powers of RKHSs and scaled RKHSs is discussed in more detail at the end of Section~\ref{sec:scaled-rkhs}. In Section~\ref{sec:examples} we demonstrate that scaled RKHSs are more useful in describing sample path properties of Gaussian processes defined by infinitely smooth kernels than powers of RKHSs.

To the best of our knowledge, the question about a ``minimal'' set which contains the samples with probability one has received only cursory discussion in the literature.
Perhaps the most relevant digression on the topic is an observation by \citet[pp.\@~369--370]{Steinwart2019}, given here in a somewhat applied form and without some technicalities, that the samples are contained in the set
\begin{equation} \label{eq:sobolev-sample-set}
  \Bigg( \bigcap_{ r < s } W_2^r(T) \Bigg) \setminus W_2^s(T)
\end{equation}
with probability one if $H(K)$ is norm-equivalent to the fractional Sobolev space $W_2^{s+d/2}(T)$ for $s > 0$ on a suitable domain $T \subset \R^d$.
In the Sobolev case the samples are therefore ``$d/2$ less smooth'' than functions in the RKHS of $K$.
Because $W_2^s(T) = \cup_{r \geq s} W_2^r(T)$, the set in~\eqref{eq:sobolev-sample-set} has the same form as the sample support set in~\eqref{eq:sample-set-difference}.
This observation is, of course, a general version of the familiar result that the sample paths of the Brownian motion, whose covariance kernel $K(t,t') = \min\{t,t'\}$ on $T = [0,1]$ induces the Sobolev space $W_2^1([0,1])$ with zero boundary condition at the origin, have regularity $1/2$ in the sense that they are almost surely $\alpha$-Hölder continuous if and only if $\alpha < 1/2$.
That is, there is $C > 0$ such that, for almost every $\omega \in \Omega$,
  $\abs[0]{X(t,\omega) - X(t',\omega)} \leq C \abs[0]{t - t'}^\alpha$
for all $t,t' \in [0,1]$ and any $\alpha < 1/2$.
However, L\'{e}vy's modulus of continuity theorem~\citep[Section~1.2]{MortersPeres2010} improves this to
  $\abs[0]{X(t,\omega) - X(t',\omega)} \leq C \sqrt{h \log (1/h)}$
when $h = \abs[0]{t-t'}$ is sufficiently small.
Since the Sobolev space $W_2^s(T)$ consists of those functions $f \colon T \to \R$ which admit an $L^2(\R^d)$-extension $f_e \colon \R^d \to \R$ whose Fourier transform satisfies
\begin{equation} \label{eq:sobolev-fourier}
  \int_{\R^d} \big(1 + \norm[0]{\xi}^2 \big)^s \abs[0]{\widehat{f}_e(\xi)}^2 \dif \xi < \infty,
\end{equation}
L\'{e}vy's modulus of continuity theorem suggests replacing the weight in~\eqref{eq:sobolev-fourier} with, for example, \sloppy{${(1+\norm[0]{\xi}^2)^s \log(1+\norm[0]{\xi})}$} so that the resulting function space is a proper superset of $W_2^s(T)$ and a proper subset of $W_2^r(T)$ for every $r < s$ and hence a proper subset of the set in~\eqref{eq:sobolev-sample-set}.
Some results and discussion in~\citep[Section~4.2]{Karvonen-MLE2020} and~\citep{Scheuerer2010} have this flavour.

Finally, we remark that classical results about the topological support of a Gaussian measure are distinct from the results in this article.
Let $C(T)$ be equipped with the standard supremum norm $\norm[0]{\cdot}_\infty$.
The \emph{topological support}, $\mathrm{supp}_{C(T)}(\mu_X)$, of the measure $\mu_X$ is the set
\begin{equation*}
  \mathrm{supp}_{C(T)}(\mu_X) = \Set{ f \in C(T) }{ \mu_X(B(f, r)) > 0 \text{ for all } r > 0},
\end{equation*}
where $B(f,r)$ is the $f$-centered $r$-ball in $(C(T), \norm[0]{\cdot}_\infty)$.
It is a classical result~\citep[Theorem~3]{Kallianpur1971} that
  $\mathrm{supp}_{C(T)}( \mu_X ) = \overline{ H(K) }$,
where $\overline{H(K)}$ is the closure of $H(K)$ in $(C(T), \norm[0]{\cdot}_\infty)$.
In other words, the topological support of $\mu_X$ contains every continuous function $f$ such that for every $\varepsilon > 0$ there exist $g_\varepsilon \in H(K)$ satisfying $\norm[0]{f-g_\varepsilon}_\infty < \varepsilon$.
Now, recall that a kernel $R$ is \emph{universal} if $H(R)$ is dense in $C(T)$~\citep[Section~4.6]{Steinwart2008}.
Most kernels of interest to practitioners are universal, including Gaussians, Matérns, and power series kernels.
But, by definition, the closure of the RKHS of a universal kernel equals $C(T)$.
Therefore
  $\mathrm{supp}_{C(T)}( \mu_X ) = \overline{ H(K) } = C(T)$
if $K$ is a universal kernel.
This result does not provide any information about the samples because we have assumed that the samples are continuous to begin with.
See~\citep[Section~3.6]{Bogachev1998} for further results on general topological supports of Gaussian measures.

\section{Scaled Reproducing Kernel Hilbert Spaces} \label{sec:scaled-rkhs}

For any orthonormal basis $\Phi = (\phi_n)_{n=1}^\infty$ of $H(K)$ the kernel has the pointwise convergent expansion $K(t, t') = \sum_{n=1}^\infty \phi_n(t) \phi_n(t')$ for all $t, t' \in T$.
By the standard characterisation of a separable Hilbert space, the RKHS consists of precisely those functions $f \colon T \to \R$ that admit an expansion $f = \sum_{n=1}^\infty f_n \phi_n$ for coefficients such that $\sum_{n=1}^\infty f_n^2 < \infty$.
The Cauchy--Schwarz inequality ensures that this expansion converges pointwise on $T$.
For given functions $f = \sum_{n=1}^\infty f_n \phi_n$ and $g = \sum_{n=1}^\infty g_n \phi_n$ in the RKHS the inner product is $\inprod{f}{g}_K = \sum_{n=1}^\infty f_n g_n$.

\begin{definition}[Scaled kernel and RKHS]
We say that a positive sequence $A = (\alpha_n)_{n=1}^\infty$ is a $\Phi$-\emph{scaling} of $H(K)$ if
$\sum_{n=1}^\infty \alpha_n \phi_n(t)^2 < \infty$ for every $t \in T$.
The kernel
\begin{equation} \label{eq:scaled-kernel}
  K_{A,\Phi}(t,t') = \sum_{n=1}^\infty \alpha_n \phi_n(t) \phi_n(t')
\end{equation}
is called a \emph{scaled kernel} and its RKHS $H(K_{A,\Phi})$ a \emph{scaled RKHS}.
\end{definition}

See~\citep{RakotomamonjyCanu2005,XuZhang2009,ZhangZhao2013} for prior appearances of scaled kernels under different names and not in the context of Gaussian processes.
Although many of the results in this section have appeared in some form in the literature, all proofs are included here for completeness.

\begin{proposition} \label{thm:scaled-RKHS}
  Let $\Phi = (\phi_n)_{n=1}^\infty$ be an orthonormal basis of $H(K)$ and $A = (\alpha_n)_{n=1}^\infty$ a $\Phi$-scaling of $H(K)$.
  Then (i) the scaled kernel $K_{A,\Phi}$ is positive-semidefinite, (ii) the collection $(\sqrt{\alpha_n} \phi_n)_{n=1}^\infty$ is an orthonormal basis of $H(K_{A,\Phi})$, and (iii) the scaled RKHS is
  \begin{equation} \label{eq:scaled-RKHS}
        H(K_{A,\Phi}) = \Set[\Bigg]{ f = \sum_{n=1}^\infty f_n \phi_n }{ \norm[0]{f}_{K_{A,\Phi}}^2 = \sum_{n=1}^\infty \frac{f_n^2}{\alpha_n} < \infty},
  \end{equation}
  where convergence is pointwise, and for any $f = \sum_{n=1}^\infty f_n \phi_n$ and $g = \sum_{n=1}^\infty g_n \phi_n$ in $H(K_{A,\Phi})$ its inner product is
    \begin{equation} \label{eq:scaled-inner-product}
      \inprod{f}{g}_{K_{A,\Phi}} = \sum_{n=1}^\infty \frac{f_n g_n}{\alpha_n}.
    \end{equation}
\end{proposition}
\begin{proof}
  By the Cauchy--Schwarz inequality and $\sum_{n=1}^\infty \alpha_n \phi_n(t)^2 < \infty$ for every $t \in T$,
  \begin{equation*}
    \sum_{n=1}^\infty \abs[0]{\alpha_n \phi_n(t) \phi_n(t')} \leq \bigg( \sum_{n=1}^\infty \alpha_n \phi_n(t)^2 \bigg)^{1/2} \bigg( \sum_{n=1}^\infty \alpha_n \phi_n(t')^2 \bigg)^{1/2} < \infty
  \end{equation*}
  for any $t, t' \in T$.
  This proves that the scaled kernel in~\eqref{eq:scaled-kernel} is well-defined via an absolutely convergent series.
  To verify that $K_{A,\Phi}$ is positive-semidefinite, note that, for any $N \geq 1$, $a_1, \ldots, a_N \in \R$, and $t_1, \ldots, t_N \in T$, 
  \begin{equation*}
      \sum_{i=1}^N \sum_{j=1}^N a_i a_j K_{A,\Phi}(t_i, t_j) = \sum_{n=1}^\infty \alpha_n \sum_{i=1}^N \sum_{j=1}^N a_i a_j \phi_n(t_i) \phi_n(t_j) = \sum_{n=1}^\infty \alpha_n \bigg( \sum_{i=1}^N a_i \phi_n(t_i) \bigg)^2
  \end{equation*}
  is non-negative because each $\alpha_n$ is positive.
  Because
  \begin{equation*}
    \sum_{n=1}^\infty \abs[0]{ f_n \sqrt{\alpha_n} \phi_n(t) } \leq \bigg( \sum_{n=1}^\infty f_n^2 \bigg)^{1/2} \bigg( \sum_{n=1}^\infty \alpha_n \phi_n(t)^2 \bigg)^{1/2} < \infty \:\: \text{ for every } \:\: t \in T
  \end{equation*}
  if $\sum_{n=1}^\infty f_n^2 < \infty$, the space defined in~\eqref{eq:scaled-RKHS} and~\eqref{eq:scaled-inner-product} is a Hilbert space of functions with an orthonormal basis $(\sqrt{\alpha_n} \phi_n)_{n=1}^\infty$.
  Since $K_{A,\Phi}(t,t') = \sum_{n=1}^\infty \alpha_n \phi_n(t) \phi_n(t')$, the scaled kernel is the unique reproducing kernel of this space~\citep[e.g.,][Theorem~9]{Minh2010}.
\end{proof}

A scaled RKHS depends on the ordering of the orthonormal basis of $H(K)$ used to construct it.
For example, let $\Phi = (\phi_n)_{n=1}^\infty$ be an orthonormal basis of $H(K)$ and suppose that $\alpha_n = n$ defines a $\Phi$-scaling of $H(K)$.
Define another ordered orthonormal basis $\Psi = (\psi_n)_{n=1}^\infty$ by setting $\psi_{2n+1} = \phi_{2^n}$ for $n \geq 0$ and interleaving the remaining $\phi_n$ to produce $\Psi = (\phi_1, \phi_3, \phi_2, \phi_5, \phi_4, \phi_6, \phi_8, \phi_{7}, \phi_{16}, \ldots )$.
The function $f = \sum_{n=0}^\infty \phi_{2^n} \eqqcolon \sum_{n=1}^\infty f_{\Phi,n} \phi_n$ is in $H(K_{A,\Phi})$ because
\begin{equation*}
  \norm[0]{f}_{K_{A,\Phi}}^2 = \sum_{n=1}^\infty \frac{f_{\Phi,n}^2}{\alpha_n} = \sum_{n=0}^\infty \frac{1}{2^n} < \infty
\end{equation*}
but not in $H(K_{A,\Psi})$ because $f = \sum_{n=0}^\infty \phi_{2^n} = \sum_{n=0}^\infty \psi_{2n+1} \eqqcolon \sum_{n=1}^\infty f_{\Psi,n} \psi_n$ and therefore
\begin{equation*}
  \norm[0]{f}_{K_{A,\Psi}}^2 = \sum_{n=1}^\infty \frac{f_{\Psi,n}^2}{\alpha_n} = \sum_{n=0}^\infty \frac{1}{2n+1} = \infty.
\end{equation*}
In practice, the orthonormal basis usually has a natural ordering. For instance, the decreasing eigenvalues specify an ordering for a basis obtained from Mercer's theorem (see Section~\ref{sec:related}) or the basis may have a polynomial factor, the degree of which specifies an ordering (see the kernels in Sections~\ref{sec:gaussian-kernel} and~\ref{sec:power-series-kernels}).

The following results compare sizes of scaled RKHSs: the faster $\alpha_n$ grows, the larger the RKHS $H(K_{A,\Phi})$ is.
A number of additional properties between scaled RKHSs can be proved in a similar manner but are not needed in the developments of this article.
Some of the below results or their variants can be found in the literature.
In particular, see~\citep[Section~6]{XuZhang2009} and~\citep[Section~4]{ZhangZhao2013} for a version of Proposition~\ref{prop:scaled-proper-subset} and some additional results.

\begin{proposition} \label{prop:scaled-relations} Let $\Phi = (\phi_n)_{n=1}^\infty$ be an orthonormal basis of $H(K)$ and $A = (\alpha_n)_{n=1}^\infty$ and $B = (\beta_n)_{n=1}^\infty$ two $\Phi$-scalings of $H(K)$. Then
  $H(K_{B,\Phi}) \subset H(K_{A,\Phi})$ if and only if $\beta_n \preceq \alpha_n$.
  In particular,
  $H(K) \subset H(K_{A,\Phi})$ if and only if $\inf_{n \geq 1} \alpha_n > 0$.
\end{proposition}
\begin{proof}
  If $\beta_n \preceq \alpha_n$, then for any $f = \sum_{n=1}^\infty f_n \phi_n \in H(K_{B,\Phi})$ we have
  \begin{equation} \label{eq:continuous-embedding}
    \norm[0]{f}_{K_{A,\Phi}}^2 = \sum_{n=1}^\infty \frac{f_n^2}{\alpha_n} = \sum_{n=1}^\infty \frac{\beta_n}{\alpha_n} \, \frac{f_n^2}{\beta_n} \leq \norm[0]{f}_{K_{B,\Phi}}^2 \sup_{n \geq 1} \frac{\beta_n}{\alpha_n} < \infty.
  \end{equation}
  Consequently, $H(K_{B,\Phi}) \subset H(K_{A,\Phi})$.
  Suppose then that $H(K_{B,\Phi}) \subset H(K_{A,\Phi})$ and assume to the contrary that $\sup_{n \geq 1} \alpha_n^{-1} \beta_n = \infty$ so that there is a subsequence $(n_m)_{m=1}^\infty$ such that $\alpha_{n_m}^{-1} \beta_{n_m} \geq 2^m$.
  Then $f = \sum_{m=1}^\infty 2^{-m/2} \sqrt{\beta_{n_m}} \phi_{n_m} \in H(K_{B,\Phi}) \setminus H(K_{A,\Phi})$ since
  \begin{equation*}
    \norm[0]{f}_{K_{B,\Phi}}^2 = \sum_{m=1}^\infty \bigg( \frac{ \sqrt{\beta_{n_m}} }{ 2^{m/2} } \bigg)^2 \frac{1}{\beta_{n_m}} = \sum_{m=1}^\infty 2^{-m} = 1 \: \text{ but } \: \norm[0]{f}_{K_{A,\Phi}}^2 = \sum_{m=1}^\infty 2^{-m} \frac{\beta_{n_m}}{\alpha_{n_m}} \geq \sum_{m=1}^\infty 1 = \infty,
  \end{equation*}
  which contradicts the assumption that $H(K_{B,\Phi}) \subset H(K_{A,\Phi})$.
  Thus $\sup_{n \geq 1} \alpha_n^{-1} \beta_n < \infty$.
  The second statement follows by setting $\beta_n = 1$ for every $n \in \N$ and noting that then $H(K_{B,\Phi}) = H(K)$.
\end{proof}

Two normed spaces $F$ and $G$ are said to be \emph{norm-equivalent} if they are equal as sets and if there exist positive constants $C_1$ and $C_2$ such that
$C_1 \norm[0]{f}_F \leq \norm[0]{f}_G \leq C_2 \norm[0]{f}_F$ for all $f \in F$.
From~\eqref{eq:continuous-embedding} it follows that $H(K_{A,\Phi})$ and $H(K_{B,\Phi})$ are norm-equivalent if and only if $\alpha_n \asymp \beta_n$.

\begin{corollary} \label{cor:norm-equivalence} Let $\Phi = (\phi_n)_{n=1}^\infty$ be an orthonormal basis of $H(K)$ and $A = (\alpha_n)_{n=1}^\infty$ and $B = (\beta_n)_{n=1}^\infty$ two $\Phi$-scalings of $H(K)$. Then $H(K_{A,\Phi})$ and $H(K_{B,\Phi})$ are norm-equivalent if and only if $\alpha_n \asymp \beta_n$.
\end{corollary}

\begin{proposition} \label{prop:scaled-proper-subset}
  Let $\Phi = (\phi_n)_{n=1}^\infty$ be an orthonormal basis of $H(K)$ and $A = (\alpha_n)_{n=1}^\infty$ and $B = (\beta_n)_{n=1}^\infty$ two $\Phi$-scalings of $H(K)$.
  Then $H(K_{B,\Phi}) \subsetneq H(K_{A,\Phi})$ if and only if $\sup_{n \geq 1} \alpha_n \beta_n^{-1} = \infty$ and $\beta_n \preceq \alpha_n$.
\end{proposition}
\begin{proof}
  Assume first that $\sup_{n \geq 1} \alpha_n \beta_n^{-1} = \infty$ and $\beta_n \preceq \alpha_n$.
  Since $\beta_n \preceq \alpha_n$, Proposition~\ref{prop:scaled-relations} yields $H(K_{B,\Phi}) \subset H(K_{A,\Phi})$.
  Thus $H(K_{B,\Phi})$ is a proper subset of $H(K_{A,\Phi})$ if $H(K_{A,\Phi})$ is not a subset of $H(K_{B,\Phi})$.
  But, again by Proposition~\ref{prop:scaled-relations}, $H(K_{A,\Phi}) \subset H(K_{B,\Phi})$ if and only if $\alpha_n \preceq \beta_n$, which contradicts the assumption that $\sup_{ n \geq 1} \alpha_n \beta_n^{-1} = \infty$.
  Hence $H(K_{B,\Phi})$ is a proper subset of $H(K_{A,\Phi})$.

  Assume then that $H(K_{B,\Phi}) \subsetneq H(K_{A,\Phi})$.
  Then $\beta_n \preceq \alpha_n$ by Proposition~\ref{prop:scaled-relations}.
  If $\sup_{n \geq 1} \alpha_n \beta_n^{-1} = \infty$ did not hold, there would exist $C > 0$ such that $\alpha_n \leq C \beta_n$ for all $n \in \N$, which is to say $\alpha_n \preceq \beta_n$.
  But by Proposition~\ref{prop:scaled-relations} this would imply that $H(K_{A,\Phi}) \subset H(K_{B,\Phi})$, which would contradict the assumption that $H(K_{B,\Phi})$ is a proper subset of $H(K_{A,\Phi})$.
  This completes the proof.
\end{proof}

\begin{remark} Let $H(R)$ be another separable RKHS of functions on $T$.
  The RKHSs $H(K)$ and $H(R)$ are \emph{simultaneously diagonalisable} if there exists an orthonormal basis $(\phi_n)_{n=1}^\infty$ of $H(K)$ which is an orthogonal basis of $H(R)$.
  That is, $(\norm[0]{\phi_n}_R^{-1} \phi_n)_{n=1}^\infty$ is an orthonormal basis of $H(R)$ and consequently $H(R) = H(K_{A,\Phi})$ for the scaling with $\alpha_n = \norm[0]{\phi_n}_R^{-2}$.
\end{remark}

We conclude this section by demonstrating that scaled RKHSs generalise powers of RKHSs.
We say that a $\Phi$-scaling $A_\rho = (\alpha_n)_{n=1}^\infty$ of $H(K)$ is $\rho$-\emph{hyperharmonic} if $\alpha_n = n^{\rho}$ for some $\rho \geq 0$. The RKHS $H(K_{A_\rho,\Phi})$ is a $\rho$-\emph{hyperharmonic scaled RKHS}.
Recall from Section~\ref{sec:related} that if $T$ is a compact metric space, $K$ is continuous on $T \times T$, and~$\nu$ is a finite and strictly positive Borel measure on $T$, then the integral operator in~\eqref{eq:mercer-operator} has eigenfunctions $(\psi_n)_{n=1}^\infty$ and decreasing positive eigenvalues $(\lambda_n)_{n=1}^\infty$ and $\Psi = (\sqrt{\lambda_n} \psi_n)_{n=1}^\infty$ is an orthonormal basis of $H(K)$.
For $\theta > 0$ the kernel \sloppy{${K^{(\theta)}(t,t') = \sum_{n=1}^\infty \lambda_n^\theta \psi_n(t) \psi_n(t')}$} is the $\theta$th power of $K$ and its RKHS $H(K^{(\theta)})$ the $\theta$th power of $H(K)$.
These objects are well-defined if $\sum_{n=1}^\infty \lambda_n^\theta \psi_n(t) < \infty$ for all $t \in T$.
We immediately recognise that $K^{(\theta)}$ equals the scaled kernel $K_{A,\Psi}$ for the scaling $A = ( \lambda_n^{\theta-1})_{n=1}^\infty$ because
\begin{equation*}
  K_{A,\Psi}(t,t') = \sum_{n=1}^\infty \lambda_n^{\theta-1} \lambda_n \psi_n(t) \psi_n(t') = \sum_{n=1}^\infty \lambda_n^{\theta} \psi_n(t) \psi_n(t') = K^{(\theta)}(t, t').
\end{equation*}
Polynomially decaying eigenvalues ($\lambda_n \asymp n^{-p}$ for some $p > 0$) are an important special case.
For example, this holds with $p=-2s/d$ if $T \subset \R^d$ and $H(K)$ is norm-equivalent to the Sobolev space $W_2^s(T)$ for $s > d/2$~\citep[p.\@~370]{Steinwart2019}.
If $\lambda_n \asymp n^{-p}$ and $\rho = p(1-\theta)$, the $\rho$-hyperharmonic scaled RKHS is norm-equivalent to the power RKHS $H(K^{(\theta)})$ by Corollary~\ref{cor:norm-equivalence} since $n^\rho \lambda_n \asymp n^{-\theta p} \asymp \lambda_n^\theta$.

\section{Sample Path Properties} \label{sec:sample-paths}

This section contains the main results of the article.
First, we consider a specialisation to scaled RKHSs of a theorem originally proved by Driscoll~\citep{Driscoll1973} and later generalised by Luk\'ic and Beder~\citep{LukicBeder2001}.
Then we define general sample support sets and characterise them for $\sigma$-algebras generated by scalings of $H(K)$.

\subsection{Domination and Generalised Driscoll's Theorem for Scaled RKHSs} \label{sec:domination}

A kernel $R$ on $T$ \emph{dominates} $K$ if $H(K) \subset H(R)$.
In this case there exists~\citep[Theorem~1.1]{LukicBeder2001} a unique linear operator $L \colon H(R) \to H(K)$, called the \emph{dominance operator}, whose range is contained in $H(K)$ and which satisfies
$\inprod{f}{g}_R = \inprod{Lf}{g}_K$ for all $f \in H(R)$ and $g \in H(K)$.
The dominance is said to be \emph{nuclear}, denoted $R \gg K$, if $H(R)$ is separable and the operator $L$ is nuclear, which is to say that
\begin{equation} \label{eq:nuclear-dominance}
  \tr(L) = \sum_{n=1}^\infty \inprod{L \psi_n}{\psi_n}_R < \infty
\end{equation}
for any orthonormal basis $(\psi_n)_{n=1}^\infty$ of $H(R)$.\footnote{A change of basis shows that $\tr(L)$ does not depend on the orthonormal basis.}

Define the pseudometric
  $d_R(t, t') = \norm[0]{R(\cdot, t) - R(\cdot, t')}_R = \sqrt{ R(t,t) - 2R(t,t') + R(t',t') }$
on $T$.
If $R$ is positive-definite, $d_R$ is a metric.
However, positive-definiteness is not necessary for $d_R$ to be a metric.
For example, the Brownian motion kernel $R(t,t') = \min\{t,t'\}$ on $T = [0,1]$ is only positive-semidefinite because $R(t,0) = 0$ for every $t \in T$ but nevertheless yields a metric because $d_R(t,t') = \sqrt{ t - 2\min\{t,t'\} + t' }$ vanishes if and only if $t = t'$.
See~\citep[Section~4]{LukicBeder2001} for more properties of $d_R$.
Injectivity of the mapping $t \mapsto R(\cdot, t)$ is equivalent to $d_R$ being a metric.

By the following theorem, a special case of the zero-one law of Kallianpur~\citep{Kallianpur1970,LePage1973} and a generalisation by \citet[Theorem~7.5]{LukicBeder2001} of an earlier result by \citet[Theorem~3]{Driscoll1973}, the nuclear dominance condition determines whether or not the samples of a Gaussian process \sloppy{${(X(t))_{t \in T} \sim \GP(0, K)}$} lie in $H(R)$.
In particular, the probability of them being in $H(R)$ is always either one or zero.

\begin{theorem}[Generalised Driscoll's Theorem] \label{thm:driscoll} Let $R$ be a continuous kernel on $T \times T$ with separable RKHS and $(X(t))_{t \in T} \sim \GP(0, K)$.
  If $d_R$ is a metric, then either
  \begin{equation*}
    \mathbb{P}\big[ X \in H(R) \big] = 0 \:\: \text{ and } \:\: R \not\gg K \quad \text{ or } \quad \mathbb{P}\big[ X \in H(R) \big] = 1 \:\: \text{ and } \:\: R \gg K.
  \end{equation*}
\end{theorem}
\begin{proof}
  Theorem~7.5 in~\citep{LukicBeder2001} is otherwise identical except that $R$ is not assumed $d_T$-continuous and the samples of $X$ are assumed $d_R$-continuous.
  However, when $R$ is $d_T$-continuous, $d_T$-continuity of the samples, which one of our standing assumptions, implies their $d_R$-continuity.
\end{proof}

Summability of the reciprocal scaling coefficients controls whether or not a scaled RKHS contains the sample paths.

\begin{lemma} \label{lemma:scaled-dR-metric}
  Let $\Phi = (\phi_n)_{n=1}^\infty$ be an orthonormal basis of $H(K)$ and $A = (\alpha_n)_{n=1}^\infty$ a $\Phi$-scaling of $H(K)$. Let $R = K_{A,\Phi}$. If $d_K$ is a metric, then so is $d_R$.
\end{lemma}
\begin{proof}
  Because $d_K$ is a metric,
    $d_K(t,t')^2 = K(t,t) - 2K(t,t') + K(t', t') = \sum_{n=1}^\infty [ \phi_n(t) - \phi_n(t')]^2$
  vanishes if and only if $t = t'$.
  Since $d_R(t, t')^2 = \sum_{n=1}^\infty \alpha_n [ \phi_n(t) - \phi_n(t')]^2$ and $\alpha_n$ are positive, we conclude that $d_R(t, t') = 0$ if and only if $t = t'$.
\end{proof}

\begin{theorem} \label{thm:driscoll-scaled}
  Let $\Phi = (\phi_n)_{n=1}^\infty$ be an orthonormal basis of $H(K)$, $A = (\alpha_n)_{n=1}^\infty$ a $\Phi$-scaling of $H(K)$, and $(X(t))_{t \in T} \sim \mathcal{GP}(0, K)$. If $K_{A,\Phi}$ is continuous and $d_K$ is a metric, then either 
  \begin{equation*}
    \mathbb{P}\big[ X \in H(K_{A,\Phi}) \big] = 0 \:\: \text{ and } \:\: \sum_{n=1}^\infty \frac{1}{\alpha_n} = \infty \quad \text{ or } \quad \mathbb{P}\big[ X \in H(K_{A,\Phi}) \big] = 1 \:\: \text{ and } \:\: \sum_{n=1}^\infty \frac{1}{\alpha_n} < \infty.
  \end{equation*}
\end{theorem}
\begin{proof}
  Assume first that the scaling is such that $H(K) \subset H(K_{A,\Phi})$.
  It is easy to verify using Proposition~\ref{thm:scaled-RKHS} that the dominance operator $L \colon H(K_{A,\Phi}) \to H(K)$ is given by
  $Lf = \sum_{n=1}^\infty f_n \alpha_n^{-1} \phi_n$ for any $f = \sum_{n=1}^\infty f_n \phi_n \in H(K_{A,\Phi})$.
  Because $(\sqrt{\alpha_n} \phi_n)_{n=1}^\infty$ is an orthonormal basis of $H(K_{A,\Phi})$ and $L(\sqrt{\alpha_n} \phi_n) = 1/\sqrt{\alpha_n}$, the nuclear dominance condition~\eqref{eq:nuclear-dominance} is
  \begin{equation*}
    \tr(L) = \sum_{n=1}^\infty \inprod[\big]{\sqrt{\alpha_n} L \phi_n}{\sqrt{\alpha_n} \phi_n}_{K_{A,\Phi}} = \sum_{n=1}^\infty \inprod{\phi_n}{\phi_n}_{K_{A,\Phi}} = \sum_{n=1}^\infty \frac{1}{\alpha_n},
  \end{equation*}
  and the claim follows from Theorem~\ref{thm:driscoll} since Lemma~\ref{lemma:scaled-dR-metric} guarantees that $d_{R}$ for $R = K_{A,\Phi}$ is a metric.
  Assume then that $H(K) \not\subset H(K_{A,\Phi})$.
  It is trivial that $K_{A,\Phi} \not\gg K$.
  Thus $\mathbb{P}[X \in H(K_{A,\Phi})] = 0$.
  If we had $\sum_{n=1}^\infty \alpha_n^{-1} < \infty$, then it would necessarily hold that $\sup_{n \geq 1} \alpha_n^{-1} < \infty$ and consequently $H(K) \subset H(K_{A,\Phi})$ by Proposition~\ref{prop:scaled-relations}, which is a contradiction. Therefore $\sum_{n=1}^\infty \alpha_n^{-1} = \infty$.
\end{proof}

\subsection{Sample Support Sets} \label{sec:sample-sets}

Theorems~\ref{thm:driscoll} and~\ref{thm:driscoll-scaled} motivate us to define the sample support set of a Gaussian process with respect to a collection of kernels as the largest set on the ``boundary'' between their induced RKHSs of probabilities one and zero.
Let $\mathfrak{R}$ be a collection of continuous kernels $R$ on $T$ for which $d_R$ is a metric and $H(\mathfrak{R})$ the corresponding set of RKHSs.
Every element of $H(\mathfrak{R})$ is a subset of $C(T)$.
By the generalised Driscoll's theorem each element of $H(\mathfrak{R})$ has $\mu_X$-measure one or zero, depending on the nuclear dominance condition.
Define the disjoint sets
\begin{equation*}
  \mathfrak{R}_1(K) = \Set{R \in \mathfrak{R} }{R \gg K} \quad \text{ and } \quad \mathfrak{R}_0(K) = \Set{R \in \mathfrak{R} }{R \not\gg K}
\end{equation*}
which partition $\mathfrak{R}$.
We assume that both $\mathfrak{R}_1(K)$ and $\mathfrak{R}_0(K)$ are non-empty and introduce the notion of a sample support set.

\begin{definition}[Sample support set] Let $\mathcal{S}(\mathfrak{R}) = \sigma(H(\mathfrak{R}))$ be the $\sigma$-algebra generated by $H(\mathfrak{R})$. The \emph{sample support set}, $S_\mathfrak{R}(K)$, of the Gaussian process $(X(t))_{t \in T} \sim \mathcal{GP}(0, K)$ with respect to $\mathfrak{R}$ is the largest subset of $C(T)$ such that 
  $S_\mathfrak{R}(K) \subset H$ for every $H \in \mathcal{S}(\mathfrak{R})$ such that $\mu_X(H) = 1$.
\end{definition}

\begin{proposition} \label{prop:sample-set-difference}
  It holds that
  \begin{equation} \label{eq:sample-set-difference}
    S_\mathfrak{R}(K) = \bigcap_{ R_1 \in \mathfrak{R}_1(K)} H(R_1) \setminus \bigcup_{ R_0 \in \mathfrak{R}_0(K) } H(R_0).
  \end{equation}
\end{proposition}
\begin{proof}
  Suppose that there is $f \in S_\mathfrak{R}(K)$ which is not contained in the set on the right-hand side of~\eqref{eq:sample-set-difference}.
  That is, we have either
  $f \notin \cap_{R_1 \in \mathfrak{R}_1(K)} H(R_1)$ or $f \in \cup_{R_0 \in \mathfrak{R}_0(K)} H(R_0)$.
  In the former case there is $R_1 \in \mathfrak{R}_1(K)$ such that $f \notin H(R_1)$. But because $\mu_X(H(R_1)) = 1$ and $S_\mathfrak{R}(K) \subset H(R_1)$ by definition, this violates the assumption that $f \in S_\mathfrak{R}(K)$.
  In the latter case there is $R_0 \in \mathfrak{R}_0(K)$ such that $f \in H(R_0)$. As $\mu_X(H(R_0)) = 0$, we have for any $R_1 \in \mathfrak{R}_1$ that \sloppy{${\mu_X( H(R_1) \setminus H(R_0)) = 1}$}. But since $f \notin H(R_1) \setminus H(R_0)$, the assumption that $f \in S_\mathfrak{R}(K)$ is again violated and we conclude that $S_\mathfrak{R}(K) \subset \cap_{ R_1 \in \mathfrak{R}_1(K)} H(R_1) \setminus \cup_{ R_0 \in \mathfrak{R}_0(K) } H(R_0)$.
  
  Since all elements of $H(\mathfrak{R})$ are either of measure zero or one, so are those of $\mathcal{S}(\mathfrak{R})$.
  It is therefore clear that $\cap_{ R_1 \in \mathfrak{R}_1(K)} H(R_1) \setminus \cup_{ R_0 \in \mathfrak{R}_0(K) } H(R_0)$ is contained in every $H \in \mathcal{S}(\mathfrak{R})$ such that $\mu_X(H) = 1$.
  Consequently, $\cap_{ R_1 \in \mathfrak{R}_1(K)} H(R_1) \setminus \cup_{ R_0 \in \mathfrak{R}_0(K) } H(R_0) \subset S_\mathfrak{R}(K)$.
  This concludes the proof.
\end{proof}

The sample support set is the largest set which is contained in every set of probability one under the law of $X$ that can be expressed in terms of countably many elementary set operations of the RKHSs $H(R)$ for $R \in \mathfrak{R}$.
The larger $\mathfrak{R}$ is, the more precisely $S_\mathfrak{R}(K)$ describes the samples of $X$. But there is an important caveat.
If $\mathfrak{R}$ is countable, the sample support set is in the $\sigma$-algebra~$\mathcal{B}$, defined in~\eqref{eq:B-sigma-algebra}, and has $\mu_X$-measure one.
However, when $\mathfrak{R}$ is uncountable and does not contain countable subsets $\mathfrak{R}_1'(K) \subset \mathfrak{R}_1(K)$ and $\mathfrak{R}_0'(K) \subset \mathfrak{R}_0(K)$ such that
\begin{equation*}
  \bigcap_{ R_1 \in \mathfrak{R}_1(K) } H(R_1) = \bigcap_{ R_1 \in \mathfrak{R}_1'(K) } H(R_1) \quad \text{ and } \quad \bigcup_{ R_0 \in \mathfrak{R}_0(K) } H(R_0) = \bigcup_{ R_0 \in \mathfrak{R}_0'(K) } H(R_0),
\end{equation*}
it cannot be easily determined if $S_\mathfrak{R}(K)$ is an element of $\mathcal{B}$.

  We are mainly interested in sample support sets with respect to $\mathfrak{R}$ which consist of all scaled kernels (and will in Theorem~\ref{thm:sample-set-all} characterise this set).
It is nevertheless conceivable that one may want to or be forced to work with less rich set of kernels---scaled or not---and with such an eventuality in mind we have introduced the more general concept of a sample support set.
If $\mathfrak{R}$ is a collection of scaled kernels, the sample support set takes a substantially more concrete form.
For this purpose we introduce the concept an approximately constant sequence, which is inspired by the results collected in~\citep[\S~41]{Knopp1951}.

\begin{definition}[Approximately constant sequence] Let $\Sigma$ be a collection of non-negative sequences. A non-negative sequence $(a_n)_{n=1}^\infty$ is said to be $\Sigma$-\emph{approximately constant} if for every \sloppy{${(b_n)_{n=1}^\infty \in \Sigma}$} the series $\sum_{n=1}^\infty b_n$ and $\sum_{n=1}^\infty a_n b_n$ either both converge or diverge.
\end{definition}

  We mention two properties of approximately constant sequences: (i) If $(a_n)_{n=1}^\infty$ and $(a_n')_{n=1}^\infty$ are two $\Sigma$-approximately constant sequences, then so is their sum. (ii) The larger $\Sigma$ is, the fewer $\Sigma$-approximately sequences there are. That is, if $\Sigma_1$ and $\Sigma_2$ are two collections of non-negative sequences such that $\Sigma_1 \subset \Sigma_2$, then a non-negative sequence is $\Sigma_1$-approximately constant if it is $\Sigma_2$-approximately constant.

For the RKHS $H(K)$ and any of its orthonormal basis $\Phi = (\phi_n)_{n=1}^\infty$ we let $\mathcal{R}(\Sigma, \Phi)$ denote the set of all functions $f = \sum_{n=1}^\infty f_n \phi_n$ such that the series converges pointwise on $T$ and $(f_n^2)_{n=1}^\infty$ is a $\Sigma$-approximately constant sequence.
The following theorem provides a crucial connection between sample support sets with respect to scaled kernels and functions defined as orthonormal expansions with approximately constant coefficients.

\begin{theorem} \label{thm:sample-sets}
  Let $\Phi = (\phi_n)_{n=1}^\infty$ be an orthonormal basis of $H(K)$ and $\Sigma_\Phi$ a collection of $\Phi$-scalings of $H(K)$ such that the corresponding scaled kernels are continuous.
  Suppose that $d_K$ is a metric and let $\mathfrak{R} = \Set{ K_{A,\Phi}}{ A \in \Sigma_\Phi }$.
  Then $S_\mathfrak{R}(K) = \mathcal{R}(\Sigma_\Phi, \Phi)$.
\end{theorem}
\begin{proof}
  Note first that, by Lemma~\ref{lemma:scaled-dR-metric}, $d_R$ is a metric for every $R \in \mathfrak{R}$.
  Because every scaling of $H(K)$ has an orthonormal basis that is a scaled version of $(\phi_n)_{n=1}^\infty$, every $f \in S_\mathfrak{R}(K)$ can be written as $f = \sum_{n=1}^\infty f_n \phi_n$ for some real coefficients $f_n$.
  Let $\Sigma_1(K)$ and $\Sigma_0(K)$ stand for the collections of $(\alpha_n)_{n=1}^\infty \in \Sigma_\Phi$ such that $\sum_{n=1}^\infty \alpha_n^{-1} < \infty$ and $\sum_{n=1}^\infty \alpha_n^{-1} = \infty$, respectively.
  Then, by Theorem~\ref{thm:driscoll-scaled}, $K_{A,\Phi} \in \mathfrak{R}_1(K)$ if $A \in \Sigma_1(K)$ and $K_{A,\Phi} \in \mathfrak{R}_0(K)$ if $A \in \Sigma_0(K)$.
  Because, by definition, \sloppy{${S_\mathfrak{R}(K) \subset H(K_{A,\Phi})}$} for any $A \in \Sigma_1(K)$ and $S_\mathfrak{R}(K) \cap H(K_{A,\Phi}) = \emptyset$ for any $A \in \Sigma_0(K)$ it follows that for every $f \in S_\mathfrak{R}(K)$ and any $(\alpha_n)_{n=1}^\infty \in \Sigma_\Phi$ we have
  \begin{equation*}
    \sum_{n=1}^\infty \frac{f_n^2}{\alpha_n} < \infty \:\: \text{ and } \:\: \sum_{n=1}^\infty \frac{1}{\alpha_n} < \infty \quad \text{ or } \quad \sum_{n=1}^\infty \frac{f_n^2}{\alpha_n} = \infty \:\: \text{ and } \:\: \sum_{n=1}^\infty \frac{1}{\alpha_n} = \infty.
  \end{equation*}
  That is, $(f_n^2)_{n=1}^\infty$ is a $\Sigma_\Phi$-approximately constant sequence and thus $S_\mathfrak{R}(K) \subset \mathcal{R}(\Sigma_\Phi, \Phi)$.
  Conversely, if $f \in \mathcal{R}(\Sigma_\Phi, \Phi)$, then $f \in H(K_{A,\Phi})$ for every $A \in \Sigma_1(K)$ and $f \notin H(K_{A,\Phi})$ for every $A \in \Sigma_0(K)$.
  Hence $f \in S_\mathfrak{R}(K)$ and thus $S_\mathfrak{R}(K) = \mathcal{R}(\Sigma_\Phi, \Phi)$.
\end{proof}

Next we use Theorem~\ref{thm:sample-sets} to describe the sample support set more concretely.

\subsection{Sample Support Sets for Scaled RKHSs} \label{sec:sample-set-all}

Let $\Sigma$ be the set of all positive sequences. Then the collection of $\Sigma$-approximately constant sequences is precisely the collection of non-negative sequences $(a_n)_{n=1}^\infty$ such that
\begin{equation} \label{eq:app-constant-1}
  \liminf_{n \to \infty} a_n > 0 \quad \text{ and } \quad \sup_{n \geq 1} a_n < \infty.
\end{equation}
For suppose that there existed a $\Sigma$-approximately constant sequence $(a_n)_{n=1}^\infty$ that violated~\eqref{eq:app-constant-1}.
If $\liminf_{n \to \infty} a_n = 0$, then there is a subsequence $(a_{n_m})_{m=1}^\infty$ such that $a_{n_m} \leq 2^{-m}$ for all $m \in \N$.
Let $(b_n)_{n=1}^\infty \in \Sigma$ be a sequence such that $b_n = 2^{-n} a_n^{-1}$ for $n \notin (n_m)_{m=1}^\infty$ and $b_{n_m} = 1$ for $m \in \N$.
Then $\sum_{n=1}^\infty b_n$ diverges but
\begin{equation*}
    \sum_{n=1}^\infty a_n b_n = \sum_{n \notin (n_m)_{m=1}^\infty} a_n b_n + \sum_{m=1}^\infty a_{n_m} b_{n_m} \leq \sum_{n \notin (n_m)_{m=1}^\infty} 2^{-n} + \sum_{m=1}^\infty 2^{-m} < \infty,
\end{equation*}
which contradicts the assumption that $(a_n)_{n=1}^\infty$ is a $\Sigma$-approximately constant sequence.
A similar argument (with $a_{n_m} \geq 2^m$, $b_n = 2^{-n}$, and $b_{n_m} = 2^{-m}$) shows the second condition in~\eqref{eq:app-constant-1} cannot be violated either; thus every $\Sigma$-approximately constant sequence satisfies~\eqref{eq:app-constant-1}.
A sequence satisfying~\eqref{eq:app-constant-1} is trivially $\Sigma$-approximately constant because the conditions imply the existence of constants $0 < c_1 \leq c_2$ such that $c_1 \leq a_n \leq c_2$ for all sufficiently large $n$.
This, together with Theorem~\ref{thm:sample-sets}, yields the following theorem which we consider the main result of this article.
The full proof is more complicated than the above argument as we cannot assume that every positive sequence is a scaling of $H(K)$.

\begin{theorem} \label{thm:sample-set-all} Let $\Phi = (\phi_n)_{n=1}^\infty$ be an orthonormal basis of $H(K)$ and suppose that there is a $\Phi$-scaling $A = (\alpha_n)_{n=1}^\infty$ of $H(K)$ such that $\sum_{n=1}^\infty \alpha_n^{-1} < \infty$ and $K_{A,\Phi}$ is continuous.
  Let $\Sigma_\Phi$ be the collection of $\Phi$-scalings of $H(K)$ such that the corresponding scaled kernels are continuous.
  Suppose that $d_K$ is a metric and let $\mathfrak{R} = \Set{K_{A,\Phi}}{A \in \Sigma_\Phi}$.
  Then $S_{\mathfrak{R}}(K)$ is non-empty and consists precisely of the functions $f = \sum_{n=1}^\infty f_n \phi_n$ such that
  \begin{equation} \label{eq:app-constant-1-f}
    \liminf_{n \to \infty} f_n^2 > 0 \quad \text{ and } \quad \sup_{n \geq 1} f_n^2 < \infty.
  \end{equation}
\end{theorem}
\begin{proof} By Theorem~\ref{thm:sample-sets}, $S_{\mathfrak{R}}(K) = \mathcal{R}(\Sigma_\Phi, \Phi)$ and the assumption that there is a scaling $A = (\alpha_n)_{n=1}^\infty$ of $H(K)$ such that $\sum_{n=1}^\infty \alpha_n^{-1} < \infty$ implies that $\mathcal{R}(\Sigma_\Phi, \Phi)$ is non-empty. Hence we have to show that functions in $\mathcal{R}(\Sigma_\Phi, \Phi)$ satisfy~\eqref{eq:app-constant-1-f}.
  Suppose that there is a function \sloppy{${f = \sum_{n=1}^\infty f_n \phi_n \in \mathcal{R}(\Sigma_\Phi, \Phi)}$} that violates~\eqref{eq:app-constant-1-f}.
  If $\sup_{n \geq 1} f_n^2 = \infty$, then there is a subsequence $(f_{n_m}^2)_{m=1}^\infty$ such that $f_{n_m}^2 \geq 2^m$ for all $m \in \N$.
  Define a sequence $B = (\beta_n)_{n=1}^\infty$ by setting \sloppy{${\beta_{n_m} = 2^m \leq f_{n_m}^2}$} for $m$ such that $f_{n_m}^2 < \alpha_{n_m}$ and $\beta_n = \alpha_n$ for all other $n$.
  Then $B$ is a $\Phi$-scaling of $H(K)$ since $\beta_n \preceq \alpha_n$ and
  \begin{equation*}
    \sum_{n=1}^\infty \beta_n^{-1} \leq \sum_{n=1}^\infty \alpha_n^{-1} + \sum_{m=1}^\infty 2^{-m} < \infty.
  \end{equation*}
  Moreover, Proposition~\ref{prop:scaled-relations} implies that $H(K_{B,\Phi}) \subset H(K_{A,\Phi})$ and it thus follows from the continuity of $K_{A,\Phi}$ that $K_{B,\Phi}$ is continuous, so that $B \in \Sigma_\Phi$.
  However,
  \begin{align*}
      \sum_{n=1}^\infty \frac{f_n^2}{\beta_n} &= \sum_{ n \notin (n_m)_{m=1}^\infty } \frac{f_n^2}{\beta_n} + \sum_{\substack{m \in \N \\ f_{n_m}^2 \geq \alpha_{n_m}}} \frac{f_{n_m}^2}{\beta_{n_m}} + \sum_{\substack{m \in \N \\ f_{n_m}^2 < \alpha_{n_m}}} \frac{f_{n_m}^2}{\beta_{n_m}} \\
      &= \sum_{ n \notin (n_m)_{m=1}^\infty } \frac{f_n^2}{\alpha_n} + \sum_{\substack{m \in \N \\ f_{n_m}^2 \geq \alpha_{n_m}}} \frac{f_{n_m}^2}{\alpha_{n_m}} + \sum_{\substack{m \in \N \\ f_{n_m}^2 < \alpha_{n_m}}} \frac{f_{n_m}^2}{2^m} \\
      &\geq \sum_{ n \notin (n_m)_{m=1}^\infty } \frac{f_n^2}{\alpha_n} + \sum_{\substack{m \in \N \\ f_{n_m}^2 \geq \alpha_{n_m}}} 1 + \sum_{\substack{m \in \N \\ f_{n_m}^2 < \alpha_{n_m}}} 1 \\
      &= \sum_{ n \notin (n_m)_{m=1}^\infty } \frac{f_n^2}{\alpha_n} + \sum_{m=1}^\infty 1 \\
      &= \infty,
  \end{align*}
  which contradicts the assumption that $(f_n^2)_{n=1}^\infty$ is $\Sigma_\Phi$-approximately constant.
  On the other hand, if $\liminf_{n \to \infty} f_n^2 = 0$, then there is a subsequence $(f_{n_m})_{m=1}^\infty$ such that $f_{n_m}^2 \leq 2^{-m}$ for all \sloppy{${m \in \N}$}.
  The sequence $B = (\beta_n)_{n=1}^\infty$ defined as $\beta_{n_m} = 1$ for $m \in \N$ and $\beta_n = \alpha_n f_n^2$ for other $n$ is a $\Phi$-scaling of $H(K)$ and $K_{B,\Phi} \in \Sigma_\Phi$ because we have proved that $\sup_{n \geq 1} f_n^2 < \infty$.
  Clearly $\sum_{n=1}^\infty \beta_n^{-1} = \infty$ but
  \begin{equation*}
    \sum_{n=1}^\infty \frac{f_n^2}{\beta_n} = \sum_{n \notin (n_m)_{m=1}^\infty} \frac{f_n^2}{\beta_n} + \sum_{m=1}^\infty \frac{f_{n_m}^2}{\beta_{n_m}} \leq \sum_{n \notin (n_m)_{m=1}^\infty} \frac{1}{\alpha_n} + \sum_{m=1}^\infty 2^{-m} < \infty,
  \end{equation*}
  which again contradicts the assumption that $(f_n^2)_{n=1}^\infty$ is $\Sigma_\Phi$-approximately constant.
\end{proof}

  Recall from Section~\ref{sec:scaled-rkhs} that a scaled RKHS depends on the ordering of the orthonormal basis of $H(K)$.
  However, the sample support set of Theorem~\ref{thm:sample-set-all} does not depend on the ordering because the characterisation~\eqref{eq:app-constant-1-f} is invariant to permutations.
  That is, let $\pi \colon \N \to \N$ be any permutation.
  Then $f = \sum_{n=1}^\infty f_n \phi_n = \sum_{n=1}^\infty f_{\pi(n)} \phi_{\pi(n)}$ but it is clear that
  \begin{equation*}
    \sup_{n \geq 1} f_n^2 = \sup_{n \geq 1} f_{\pi(n)}^2 \quad \text{ and } \quad \liminf_{n \to \infty} f_n^2 = \liminf_{n \to \infty} f_{\pi(n)}^2,
  \end{equation*}
  so that which of the bases $(\phi_n)_{n=1}^\infty$ and $(\phi_{\pi(n)})_{n=1}^\infty$ the function $f$ is expanded in does not matter.

Theorem~\ref{thm:not-countable} below demonstrates that the sample support set of Theorem~\ref{thm:sample-set-all} cannot be expressed in terms of countably many elementary set operations of scaled RKHSs---and thus could be non-measurable.
The reason for this is that there does not exist a useful notion of a boundary between convergent and divergent series, a result which is partially contained in Lemma~\ref{lemma:no-series-boundary}.
The lemma is a modified version of a result originally due to du Bois-Reymund and Hadamard~\citep[pp.\@~301--302]{Knopp1951}.
In its proof a classical result by Dini~\citep[p.\@~293]{Knopp1951} is needed.

\begin{lemma}[Dini] \label{lemma:dini} Let $(a_n')_{n=1}^\infty$ be a positive sequence such that $\sum_{n=1}^\infty a_n' < \infty$ and set
  \begin{equation*}
    a_n = \frac{a_n'}{( \sum_{l=n}^\infty a_l' )^{c}}.
  \end{equation*}
  Then $\lim_{n \to \infty} (a_n')^{-1} a_n = \infty$ and the series $\sum_{n=1}^\infty a_n$ converges if and only if $0 \leq c < 1$.
\end{lemma}

\begin{lemma} \label{lemma:no-series-boundary} For each $m \in \N$, suppose that $(a_{m,n})_{n=1}^\infty$ is a positive sequence such that \sloppy{${\sum_{n=1}^\infty a_{m,n} < \infty}$}.
  Then there is a positive sequence $(a_n)_{n=1}^\infty$ such that
  \begin{equation*}
    \sum_{n=1}^\infty a_n < \infty \quad \text{ and } \quad \lim_{n \to \infty} \frac{a_n}{a_{m,n}} = \infty \quad \text{ for every $m \in \N$.}
  \end{equation*}
\end{lemma}
\begin{proof} For each $m \in \N$, define the sequence $(\bar{a}_{m,n})_{n=1}^\infty$ by setting $\bar{a}_{m,n} = \sum_{k=1}^m a_{k,n}$.
  Therefore $\bar{a}_{m,n} \leq \bar{a}_{m+1,n}$ and $\sum_{n=1}^\infty \bar{a}_{m,n} = \sum_{k=1}^m \sum_{n=1}^\infty a_{k,n} < \infty$ for every $m$.
  Consequently, there is a strictly increasing sequence $(n_m)_{m=1}^\infty$ such that $\sum_{n=n_m}^\infty \bar{a}_{m,n} \leq 2^{-m}$ for every $m \geq 2$.
  Set $a_n' = \bar{a}_{1,n}$ when $n < n_2$ and $a_n' = \bar{a}_{m,n}$ when $n_m \leq n < n_{m+1}$ for $m \geq 2$.
  Then
  \begin{equation*}
    \sum_{n=1}^\infty a_n' = \sum_{n=1}^{n_2-1} \bar{a}_{1,n} + \sum_{m=2}^\infty \sum_{n=n_m}^{n_{m+1}-1} \bar{a}_{m,n} \leq \sum_{n=1}^{\infty} \bar{a}_{1,n} + \sum_{m=2}^\infty 2^{-m} < \infty
  \end{equation*}
  and, because $\bar{a}_{m,n} \leq \bar{a}_{m+1,n}$ and $\bar{a}_{m,n} \geq a_{m,n}$, we have $a_n' \geq \bar{a}_{m,n} \geq a_{m,n}$ for all $m$ and $n \geq n_m$.
  Finally, selecting 
  $a_n = a_n' / (\sum_{l=n}^\infty a_l' )^{1/2}$
  yields, by Lemma~\ref{lemma:dini}, a convergent series such that
  \begin{equation*}
    \lim_{n \to \infty} \frac{a_n}{a_{m,n}} \geq \lim_{n \to \infty} \frac{a_n}{a_n'} = \lim_{n \to \infty} \bigg( \sum_{l=n}^\infty a_n' \bigg)^{-1/2} = \infty
  \end{equation*}
  for every $m \in \N$.
  This proves the claim.
\end{proof}

\begin{theorem} \label{thm:not-countable}
  Let $\Phi = (\phi_n)_{n=1}^\infty$ be an orthonormal basis of $H(K)$ and suppose that there is a $\Phi$-scaling $A = (\alpha_n)_{n=1}^\infty$ of $H(K)$ such that $\sum_{n=1}^\infty \alpha_n^{-1} < \infty$ and $K_{A,\Phi}$ is continuous.
  Let $\Sigma_\Phi$ be the collection $\Phi$-scalings of $H(K)$ such that the corresponding scaled kernels are continuous.
  Suppose that $d_K$ is a metric and let $\mathfrak{R} = \Set{K_{A,\Phi}}{A \in \Sigma_\Phi}$.
  For each $m \in \N$, let $A_m = (\alpha_{m,n})_{n=1}^\infty$ and $B_m = (\beta_{m,n})_{n=1}^\infty$ be elements of $\Sigma_\Phi$ such that 
  \begin{equation*}
    \sum_{n=1}^\infty \frac{1}{\alpha_{m,n}} < \infty \quad \text{ and } \quad \sum_{n=1}^\infty \frac{1}{\beta_{m,n}} = \infty.
  \end{equation*}
  If $(X(t))_{t \in T} \sim \mathcal{GP}(0, K)$, then there exists $F \in \mathcal{B}$ such that $\mu_X(F) = 1$ and 
  \begin{equation*}
    S_{\mathfrak{R}}(K) \subsetneq F \subsetneq \bigcap_{m=1}^\infty H(K_{A_m,\Phi}) \setminus \bigcup_{m=1}^\infty H(K_{B_m,\Phi}).
  \end{equation*}
\end{theorem}
\begin{proof} By Lemma~\ref{lemma:no-series-boundary}, there is a positive sequence $A'=(\alpha_n')_{n=1}^\infty$ such that
  \begin{equation*}
    \sum_{n=1}^\infty \frac{1}{\alpha_n'} < \infty \quad \text{ and } \quad \lim_{n \to \infty} \frac{\alpha_{m,n}}{\alpha_n'} = \infty \quad \text{ for every $m \in \N$}.
  \end{equation*}
  By Proposition~\ref{prop:scaled-proper-subset}, $H(K_{A',\Phi}) \subsetneq H(K_{A_m,\Phi})$ for every $m$.
  Therefore $H(K_{A',\Phi}) \subset \bigcap_{m=1}^\infty H(K_{A_m,\Phi})$.
  Using Lemma~\ref{lemma:dini} and Proposition~\ref{prop:scaled-proper-subset} we can construct a $\Phi$-scaling $A = (\alpha_n)_{n=1}^\infty$ such that $\sum_{n=1}^\infty \alpha_n^{-1} < \infty$ and $H(K_{A,\Phi}) \subsetneq H(K_{A',\Phi})$.
  By Theorem~\ref{thm:driscoll-scaled}, the set
  \begin{equation*}
    F = H(K_{A,\Phi}) \setminus \bigcup_{m=1}^\infty H(K_{B_m,\Phi}) \subsetneq \bigcap_{m=1}^\infty H(K_{A_m,\Phi}) \setminus \bigcup_{m=1}^\infty H(K_{B_m,\Phi})
  \end{equation*}
  is of $\mu_X$-measure one and thus $S_\mathfrak{R}(K) \subset F$ by definition.
  But Proposition~\ref{prop:scaled-proper-subset} can be used again to construct a set $F' \subsetneq F$ such that $S_\mathfrak{R}(K) \subset F'$.
  Hence $S_\mathfrak{R}(K) \subsetneq F$.
  This completes the proof.
\end{proof}

In Theorem~\ref{thm:not-countable} the set $F$ was constructed by finding a scaled RKHS which is smaller than each $H(K_{A_m,\Phi})$ while still containing the sample set.
A natural question is if we can also construct a scaled RKHS that contains each $H(K_{B_m,\Phi})$ but not the sample support set.
Although Lemma~\ref{lemma:no-series-boundary} has a counterpart for divergent series, it seems difficult to guarantee that the resulting sequence is a scaling.
Namely, it can be shown that if $(\beta_{m,n})_{n=1}^\infty$ in Theorem~\ref{thm:not-countable} are such that
$\lim_{n \to \infty} \beta_{m,n} / \beta_{m+1,n} = 0$ for every $m \in \N$,
then there exists $B = (\beta_{n})_{n=1}^\infty$ such that
$\sum_{n=1}^\infty \beta_n^{-1} = \infty$ and $\lim_{n \to \infty} \beta_{m,n} / \beta_n = 0$ for every $m \in \N$.
Although there is a $\Phi$-scaling of $H(K)$ such that $\sum_{n=1}^\infty \alpha_n^{-1} < \infty$, the divergence of $\sum_{n=1}^\infty \beta_n^{-1}$ does not have to imply that $B$ is a $\Phi$-scaling because \emph{a divergent series can have terms arbitrarily larger than those of a convergent series}~\citep[p.\@~303]{Knopp1951}: there is a positive sequence $(\gamma_n)_{n=1}^\infty$ such that
$\sum_{n=1}^\infty \gamma_n^{-1} < \infty$ and $\liminf_{n \to \infty} \gamma_n / \beta_n = 0$.
If $(\gamma_n)_{n=1}^\infty$ is a $\Phi$-scaling of $H(K)$, we cannot directly deduce that so is $B$ because $\beta_n \preceq \gamma_n$ fails.

\subsection{On Monotonely Scaled RKHSs} \label{sec:sample-set-monotone}

A scaling $A = (\alpha_n)_{n=1}^\infty$ is monotone if $(\alpha_n)_{n=1}^\infty$ is a monotone sequence.
In this section we demonstrate that the characterisation~\eqref{eq:app-constant-1-f} fails if only \emph{monotone} scalings are permitted.
Let $\Sigma_\textsc{m}$ be the collection of monotone $\Phi$-scalings of $H(K)$ and $\mathfrak{R}_\textsc{m}$ the collection of corresponding scaled kernels.
Then the sample support set $S_{\mathfrak{R}_\textsc{m}}(K)$ is strictly larger than the sample support set $S_{\mathfrak{R}}(K)$ characterised by Theorem~\ref{thm:sample-set-all} because, as we show below, a $\Sigma_\textsc{m}$-approximately constant sequence can be unbounded, as long as it does not grow too rapidly, and contain arbitrarily long sequences of zeros.

Let $(g_m)_{m=0}^\infty$ be a strictly increasing sequence of positive integers and set \sloppy{${a_{g_m} = g_{m+1}-g_m}$} for $m \geq 0$ and $a_n = 0$ otherwise. Then
\begin{equation} \label{eq:cauchy-decomposition}
  \sum_{n=1}^\infty a_n b_n = \sum_{m=0}^\infty (g_{m+1} - g_m) b_{g_m}
\end{equation}
for any sequence $(b_n)_{n=1}^\infty$.
Schlömilch's generalisation of the Cauchy condensation test~\citep[p.\@~121]{Knopp1951} implies that, under the assumption $g_{m+1}-g_m \leq C(g_m - g_{m-1})$ for some $C > 0$ and all $m \geq 1$, for any non-increasing positive sequence $(b_n)_{n=1}^\infty$ the series $\sum_{n=1}^\infty b_n$ and $\sum_{m=0}^\infty (g_{m+1} - g_m) b_{g_m}$ either both converge or diverge.
That $(a_n)_{n=1}^\infty$ is $\Sigma_\textsc{m}$-approximately constant follows then from~\eqref{eq:cauchy-decomposition}.
The condition \sloppy{${g_{m+1}-g_m \leq C(g_m - g_{m-1})}$}, or equivalently $a_{g_m} \leq C a_{g_{m-1}}$, guarantees that $a_{g_m}$ does not grow too fast in relation to $g_m$.
The canonical example is obtained by setting $g_m = 2^m$ so that $a_{2^m} = 2^m$ and $C=2$ and $\sum_{m=0}^\infty (g_{m+1} - g_m) b_{g_m} = \sum_{m=0}^\infty 2^m b_{2^m}$ converges or diverges with $\sum_{n=1}^\infty b_n$ by the standard Cauchy condensation test.
Note that, because
\begin{equation*}
  \frac{1}{g_m} \sum_{n=1}^{g_m} a_{n} = \frac{1}{g_m}\sum_{k=0}^{m} (g_{k+1}-g_k) = \frac{g_{m+1}-g_0}{g_m} \leq \frac{g_{m+1}}{g_m} \leq \frac{(1+C)g_m}{g_m} = 1+C
\end{equation*}
and
\begin{equation*}
  \frac{1}{g_m-1} \sum_{n=1}^{g_m-1} a_{n} = \frac{1}{g_m - 1}\sum_{k=0}^{m-1} (g_{k+1}-g_k) = \frac{g_m - g_0}{g_m - 1},
\end{equation*}
the elements of the sequence $(a_n)_{n=1}^\infty$ are, on average, constants:
\begin{equation*}
  \liminf_{N \to \infty} \, \frac{1}{N} \sum_{n=1}^N a_n = 1 \quad \text{ and } \quad \sup_{N \geq 1} \, \frac{1}{N} \sum_{n=1}^N a_n \leq 1 + C.
\end{equation*}
This can be interpreted as a weaker form of~\eqref{eq:app-constant-1-f}.

Although $\Sigma_\textsc{m}$ contains unbounded sequences and sequences with zero lower limit, no sequence in this set can have infinity or zero as its limit.
This can be shown by using results in~\citep[§~41]{Knopp1951} which, given a sequence such that $\lim_{n \to \infty} a_n = \infty$, guarantee the existence of a monotone sequence $(c_n)_{n=1}^\infty$ such that $\sum_{n=1}^\infty c_n < \infty$ but $\sum_{n=1}^\infty a_n c_n = \infty$.
Conversely, given a sequence such that $\lim_{n \to \infty} a_n = 0$ it is possible to construct a monotone sequence $(d_n)_{n=1}^\infty$ such that $\sum_{n=1}^\infty d_n = \infty$ but $\sum_{n=1}^\infty a_n d_n < \infty$.

\subsection{On Hyperharmonic Scalings and Iterated Logarithms} \label{sec:iterated-logarithms}

We conclude this section with two general constructions before moving onto concrete examples.
Recall from Section~\ref{sec:scaled-rkhs} that scalings of the form $A_\rho = (n^\rho)_{n=1}^\infty$ for any $\rho \geq 0$ are called hyperharmonic scalings.
Because $\sum_{n=1}^\infty n^{-\rho} < \infty$ if and only if $\rho > 1$, Theorem~\ref{thm:driscoll-scaled} implies that
$\mathbb{P}\big[ X \in H(K_{A_\rho,\Phi}) \big] = 1$ if $\rho > 1$ and $\mathbb{P}\big[ X \in H(K_{A_\rho,\Phi}) \big] = 0$ if $\rho \leq 1$.
Therefore the samples are ``almost'' contained in $H(K_{A_1,\Phi})$, while $H(K_{A_{1+\varepsilon},\Phi})$ is a ``small'' RKHS which contains the samples for any ``small'' $\varepsilon > 0$.
Thus the sample support set in Theorem~\ref{thm:not-countable} satisfies
\begin{equation*}
  S_\mathfrak{R}(K) \subsetneq \bigcap_{k=1}^\infty H(K_{A_{1 + 1/k},\Phi}) \setminus H(K_{A_1,\Phi}),
\end{equation*}
where the set on the right-hand side is in $\mathcal{B}$ and has $\mu_X$-measure one.
But Theorem~\ref{thm:not-countable} also guarantees the existence of a set $F \in \mathcal{B}$ which is a proper subset of $\bigcap_{k=1}^\infty H(K_{A_{1 + 1/k},\Phi}) \setminus H(K_{A_1,\Phi})$ while having $\mu_X$-measure one.
One such set is $F = H(K_{A,\Phi}) \setminus H(K_{A',\Phi})$ for the scalings $A = (\alpha_n)_{n=1}^\infty$ and $A' = (\alpha_n')_{n=1}^\infty$ defined by $\alpha_n = n \log (n+1)^2$ and $\alpha_n' = n \log(n+1)$.
This is because $\sum_{n=1}^\infty \alpha_n^{-1}$ converges but $\sum_{n=1}^\infty (\alpha_n')^{-1}$ does not and
\begin{equation*}
  \lim_{n \to \infty} \frac{n^\rho}{n \log (n+1)^2} = \infty \: \text{ for any } \: \rho > 1 \quad \text{ and } \quad \lim_{n \to \infty} \frac{n \log (n+1)}{n} = \infty,
\end{equation*}
which by Proposition~\ref{prop:scaled-proper-subset} imply that $H(K_{A,\Phi}) \subset \bigcap_{k=1}^\infty H(K_{A_{1 + 1/k},\Phi})$ and $H(K_{A_1,\Phi}) \subsetneq H(K_{A,\Phi})$.

But one can construct even smaller measurable sets which contain the samples by the use of iterated logarithms.
The iterated logarithm $\log_p x$ is defined recursively as $\log_p x = \log ( \log_{p-1} x)$ for $p \in \N$ and $\log_0 x = x$.
  Let $p \in \N_0$ and suppose that $q \geq 0$ is large enough that $\log_p (1 + q)$ is positive.
  For any $\rho > 1$, define $A_{\textup{log}(\rho)} = (\alpha_n)_{n=1}^\infty$ by
  $\alpha_n = (n + q) \log (n + q) \times \cdots \times \log_{p-1} (n + q) \log_p (n + q)^\rho$
  and $A_\textup{log} = (\alpha_n')_{n=1}^\infty$ by
  $\alpha_n' = (n + q) \log (n + q) \times \cdots \times \log_{p-1} (n + q) \log_p (n+q)$.
  It can be proved that $\sum_{n=1}^\infty \alpha_n^{-1} < \infty$ but $\sum_{n=1}^\infty (\alpha_n')^{-1} = \infty$~\citep[pp.\@~123, 280, 293]{Knopp1951}.
  Moreover,
  \begin{equation*}
    \lim_{n \to \infty} \frac{n \log (n+1)^2}{ (n + q) \log (n + q) \times \cdots \times \log_{p-1} (n + q) \log_p (n + q)^\rho } = \infty
  \end{equation*}
  and
  \begin{equation*}
    \lim_{n \to \infty} \frac{(n + q) \log (n + q) \times \cdots \times \log_{p-1} (n + q) \log_p (n+q)}{ n \log (n + 1) } = \infty
  \end{equation*}
  if $p \geq 2$.
  Therefore iterated logarithms can be used to construct sets of $\mu_X$-measure one which are smaller than the set $F$ above.
  But Theorem~\ref{thm:not-countable} again demonstrates that for any set of $\mu_X$-measure one constructed out of scalings $A_{\textup{log}(\rho)}$ and $A_\textup{log}$ there is a smaller set which still has measure one (and which contains the sample support set).

\section{Examples} \label{sec:examples}

This section contains examples of kernels to which Theorem~\ref{thm:driscoll-scaled} can be applied to construct ``small'' RKHSs that contain the samples and ``slightly smaller'' RKHSs that do not.
For simplicity we let the domain $T$ be a finite interval on the real line and occasionally index the orthonormal bases starting from zero.
Note that to use our results one needs to have access to an orthonormal expansion of the kernel.
This rules out examples involving the popular Matérn kernels because, to the best of our knowledge, no orthonormal expansions have been computed for these kernels.
Previous results based on powers of RKHSs and convolution kernels are easier to apply in this regard, but are less flexible and expressive.
See~\citep[Section~4.4]{Kanagawa2018} and~\citep{Steinwart2019} and for examples featuring powers of RKHSs (in particular for Sobolev kernels and the Gaussian kernel) and \citep[Appendices~A.2 and~A.3]{Flaxman2016} for convolution kernel examples. \citet{Lukic2004} has examples involving integrated kernels.

\subsection{Iterated Brownian Bridge Kernels} \label{sec:iterated-BM}

Let $T = [0,1]$ and consider the \emph{iterated Brownian bridge kernel} of integer order $s \geq 2$~\citep[Section~4.1]{Cavoretto2015}:
\begin{equation} \label{eq:iterated-BB}
  K^s(t, t') = \frac{2}{\pi^{s}} \sum_{n=1}^\infty \frac{\sin(\pi n t) \sin(\pi n t')}{n^{s}} = \sum_{n=1}^\infty \phi_n(t) \phi_n(t'),
\end{equation}
where $\phi_n(t) = \sqrt{2} (\pi n)^{-s/2} \sin(\pi n t)$.
One can show that for even parameters the kernel is
\begin{equation*}
  K^{2s}(t, t') = (-1)^{s-1} \frac{2^{2s-1}}{(2s)!} \Bigg[ \mathrm{B}_{2s}\bigg(\frac{\abs[0]{t-t'}}{2}\bigg) -  \mathrm{B}_{2s}\bigg(\frac{t+t'}{2}\bigg) \Bigg],
\end{equation*}
where $\mathrm{B}_p$ is the Bernoulli polynomial of degree $p$. For $s = 2$ we obtain the Brownian bridge kernel
  $K^2(t, t') = \min\{t, t'\} - tt'$.
For $s \geq 1$ and fixed $t' \in [0,1]$ the even-order translates $K^{2s}(\cdot, t')$ are piecewise polynomials of order $2s-1$.

Iterated Brownian bridge kernels are natural candidates for hyperharmonic scalings because the $\rho$-hyperharmonic scaling $A_{\rho} = (n^{\rho})_{n=1}^\infty$ for $\rho \in \N$ such that $\rho \leq s-2$ gives
\begin{equation} \label{eq:iterated-BB-harmonic-scaling}
  K_{A_{\rho},\Phi}^s(t, t') = \frac{2}{\pi^{s}} \sum_{n=1}^\infty \frac{\sin(\pi n t) \sin(\pi n t')}{n^{s-\rho}} = \frac{1}{\pi^{\rho}}K^{s-\rho}(t,t').
\end{equation}
By Theorem~\ref{thm:driscoll-scaled}, samples of $(X(t))_{t \in T} \sim \mathcal{GP}(0, K^s)$ are in the RKHS of this kernel for any such $\rho \geq 2$ because $\sum_{n=1}^\infty n^{-\rho} < \infty$ whenever $\rho > 1$.
From the identity
  $2\int_0^1 \sin(\pi n t) \sin(\pi m t) \dif t = \delta_{nm}$
we see that the Mercer expansion with respect to the Lebesgue measure on $[0,1]$ of $K^s$ is
\begin{equation*}
  K^s(t,t') = \sum_{n=1}^\infty \lambda_n \psi_n(t) \psi_n(t') \: \text{ with } \: \lambda_n = \frac{1}{(\pi n)^{s}} \: \text{ and } \: \psi_n(t) = \sqrt{2}\sin(\pi n t).
\end{equation*}
Therefore the $\theta$th power of $H(K^s)$, with $\theta = 1-\rho/s$, equals $H(K_{A_{\rho},\Phi}^s)$ as a set.
In this case the power RKHS is recovered as an instance of a scaled RKHS.
We can also consider logarithmic scalings (recall Section~\ref{sec:iterated-logarithms}), which do not correspond to power RKHSs, such as $A = (\alpha_n)_{n=1}^\infty$ for $\alpha_n = n \log(n+1)^2$ and $A' = (\alpha_n')_{n=1}^\infty$ for $\alpha_n' = n \log (n+1)$. 
These yield the scaled kernels
\begin{equation} \label{iterated-BB-scaled-1}
  K_{A,\Phi}^s(t,t') = \frac{2}{\pi^{s}} \sum_{n=1}^\infty \frac{\sin(\pi n t) \sin(\pi n t')}{n^{s-1}} \log(n+1)^2
\end{equation}
and
\begin{equation} \label{iterated-BB-scaled-2}
  K_{A',\Phi}^s(t,t') = \frac{2}{\pi^{s}} \sum_{n=1}^\infty \frac{\sin(\pi n t) \sin(\pi n t')}{n^{s-1}} \log(n+1), 
\end{equation}
which do not appear to have closed form expressions. Because 
$\sum_{n=1}^\infty 1/(n \log(n+1)^2) < \infty$ but $\sum_{n=1}^\infty 1/(n \log(n+1)) = \infty$
the samples from $(X(t))_{t \in T} \sim \mathcal{GP}(0, K^s)$ are located in $H(K_{A,\Phi}^s)$ but not in $H(K_{A',\Phi}^s)$.

\begin{figure}[t]
  \centering
  \includegraphics[width=0.8\textwidth]{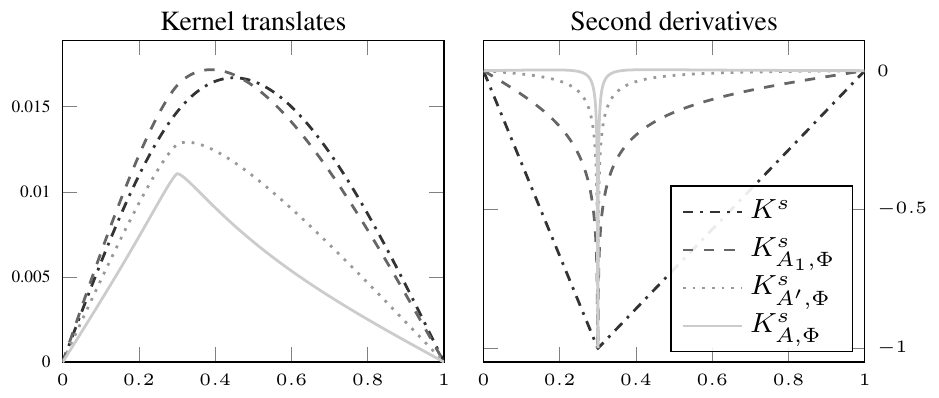}
  \caption{Translates (left) and their second derivatives (right) at $t' = 0.3$ of the kernel $K^s$ and its three translates $K_{A_1,\Phi}^s$, $K_{A',\Phi}^s$ and $K_{A,\Phi}^s$ for $s=4$. The second derivatives have been scaled so as to have $-1$ as their minimum. The scaled kernels were computed by truncating the series after 5,000 terms.} \label{fig:iterated-bb}
\end{figure}

Translates at $t' = 0.3$ of the kernel $K^s$ in~\eqref{eq:iterated-BB} and the scaled kernels $K_{A_1,\Phi}^s$ in~\eqref{eq:iterated-BB-harmonic-scaling}, $K_{A,\Phi}^s$ in~\eqref{iterated-BB-scaled-1} and $K_{A',\Phi}^s$ in~\eqref{iterated-BB-scaled-2} are plotted in Figure~\ref{fig:iterated-bb} for $s = 4$.
Also plotted are the second derivatives of the translates.
As noted, the samples from $(X(t))_{t \in T} \sim \mathcal{GP}(0, K^s)$ are in $H(K_{A,\Phi}^s)$ but not in $H(K_{A',\Phi}^s)$ or $H(K_{A_1,\Phi}^s)$.
Moreover,
\begin{equation} \label{eq:iterated-BB-inclusions}
 H(K^s) \subsetneq H(K_{A_1,\Phi}^s) \subsetneq H(K_{A',\Phi}^s) \subsetneq H(K_{A,\Phi}^s)
\end{equation}
by Proposition~\ref{prop:scaled-proper-subset}.
The second derivative of $K^s(\cdot, t')$ is Lipschitz while those of the scaled kernels are less well-behaved, though nevertheless continuous.
How the second derivatives behave is indicative of the inclusions in~\eqref{eq:iterated-BB-inclusions}: the larger the RKHS, the more severe the non-differentiability at $t=0.3$.

\subsection{Gaussian Kernel} \label{sec:gaussian-kernel}

The ubiquitous Gaussian kernel with a length-scale parameter $\ell > 0$ is
\begin{equation} \label{eq:gauss-kernel}
  K(t, t') = \exp\bigg(\!-\frac{(t-t')^2}{2\ell^2} \bigg) = \exp\bigg(\!-\frac{t^2}{2\ell^2} \bigg) \exp\bigg(\!-\frac{(t')^2}{2\ell^2} \bigg) \sum_{n=0}^\infty \frac{1}{\ell^{2n} n!} (tt')^n
\end{equation}
and the functions
\begin{equation} \label{eq:gauss-basis}
  \phi_n(t) = \frac{1}{\ell^{n} \sqrt{n!}} t^{n} \exp\bigg(\!-\frac{t^2}{2\ell^2} \bigg) \quad \text{ for } \quad n \geq 0
\end{equation}
form an orthonormal basis of its RKHS~\citep{Minh2010, Steinwart2006}.

The easiest way to proceed is to consider hyperharmonic scalings \sloppy{${A_\rho = (n^\rho)_{n=0}^\infty}$} for $\rho > 0$ with the convention $0^\rho = 1$.
Then
\begin{equation} \label{eq:gauss-kernel-hyperharmonic}
  K_{A_\rho,\Phi}(t,t') = \exp\bigg(\!-\frac{t^2+(t')^2}{2\ell^2} \bigg) \bigg( 1 + \sum_{n=1}^\infty \frac{n^\rho}{\ell^{2n} n!} (tt')^{n} \bigg).
\end{equation}
For $\rho = 1$ we get a simple analytic expression for the scaled kernel:
\begin{equation} \label{eq:gauss-kernel-harmonic}
  \begin{split}
    K_{A_1,\Phi}(t,t') &= \exp\bigg(\!-\frac{t^2+(t')^2}{2\ell^2} \bigg) \bigg( 1 + \frac{tt'}{\ell^2} \sum_{n=0}^\infty \frac{1}{n!} \bigg(\frac{tt'}{\ell^2}\bigg)^{n} \bigg) \\
    &= \exp\bigg(\!-\frac{t^2+(t')^2}{2\ell^2} \bigg) \bigg( 1 + \frac{tt'}{\ell^2} \exp\bigg( \frac{tt'}{\ell^2} \bigg) \bigg).
    \end{split}
\end{equation}
Since the harmonic series $\sum_{n=1}^\infty n^{-1}$ is the prototypical example of an ``almost'' convergent series, we can informally say that the RKHS of~\eqref{eq:gauss-kernel-harmonic} is only slightly too small to contain samples of the Gaussian process $(X(t))_{t \in T} \sim \mathcal{GP}(0, K)$ with covariance kernel~\eqref{eq:gauss-kernel}.
For a few larger values of $\rho$ we have
\begin{align*}
  K_{A_2,\Phi}(t,t') &= \exp\bigg(\!-\frac{t^2+(t')^2}{2\ell^2} \bigg) \big[ 1 + (a+a^2) \exp(a) \big], \\
  K_{A_3,\Phi}(t,t') &= \exp\bigg(\!-\frac{t^2+(t')^2}{2\ell^2} \bigg) \big[ 1 + (a+3a^2+a^3) \exp(a) \big], \\  
  K_{A_4,\Phi}(t,t') &= \exp\bigg(\!-\frac{t^2+(t')^2}{2\ell^2} \bigg) \big[ 1 + (a+7a^2+6a^3+a^4) \exp(a) \big],
\end{align*}
where $a = tt'/\ell^2$.
By Theorem~\ref{thm:driscoll-scaled}, the RKHSs of the above kernels contain the samples because \sloppy{${\sum_{n=1}^\infty n^{-\rho} < \infty}$} if and only if $\rho > 1$.

Another example can be constructed using the scaling $A = (\tau^{2n})_{n=0}^\infty$ for some $\tau > 0$. Then
\begin{equation} \label{eq:gauss-kernel-tau}
  K_{A,\Phi}(t,t') = \exp\bigg(\!-\frac{t^2+(t')^2}{2\ell^2} \bigg) \sum_{n=0}^\infty \frac{1}{n!} \bigg( \frac{\tau^2 tt'}{\ell^2}\bigg)^{n} = \exp\bigg(\!-\frac{t^2+(t')^2}{2\ell^2} + \frac{\tau^2 tt'}{\ell^2} \bigg).
\end{equation}
The RKHS of this kernel contains the samples if and only if $\sum_{n=0}^\infty \tau^{-2n}$ converges, which is equivalent to $\tau > 1$.
Since $K_{A,\Phi}$ equals the original kernel if $\tau=1$, this class of scalings is not particularly expressive.
Because the Mercer eigenvalues of the Gaussian kernel have an exponential decay this example is reminiscent (but more explicit) of the power RKHS consruction for the Gaussian kernel in Section~4.4 of \citet{Kanagawa2018}.
Observe also that with the selection $\ell^2 = (1-r^2)/r^2$ and $\tau^2 = 1/r$ for $r \in (0,1)$ the kernel equals the scaled Mehler kernel
\begin{equation*}
  K_M^r(t, t') = \exp\bigg( \! - \frac{r^2(t^2 + (t')^2) - 2r t t'}{2(1-r^2)} \bigg) = \sqrt{1-r^2} \sum_{n=0}^\infty \frac{r^n}{n!} \mathrm{H}_n(t) \mathrm{H}_n(t'),
\end{equation*}
where $\mathrm{H}_n$ is the $n$th probabilists' Hermite polynomial.
The RKHS the Mehler kernel is analysed in~\citep{IrrgeherLeobacher2015}.

A few of the kernels mentioned above are shown in Figure~\ref{fig:gaussian}. By Proposition~\ref{prop:scaled-proper-subset}, their RKHSs satisfy
  \begin{equation*}
    H(K) \subsetneq \underset{\rho\,=\,1}{H(K_{A_\rho, \Phi})} \subsetneq \underset{\rho\,=\,1.1}{H(K_{A_\rho, \Phi})} \subsetneq \underset{\rho\,=\,2}{H(K_{A_\rho, \Phi})} \subsetneq \underset{A\,=\,(\tau^{2n})_{n=0}^\infty}{H(K_{A, \Phi})}
  \end{equation*}
  and the three largest of these contain the samples of $(X(t))_{t \in T} \sim \mathcal{GP}(0, K)$.
  All the kernels are qualitatively quite similar, being all infinitely differentiable.

\begin{figure}[t]

  \centering
  \includegraphics[width=0.8\textwidth]{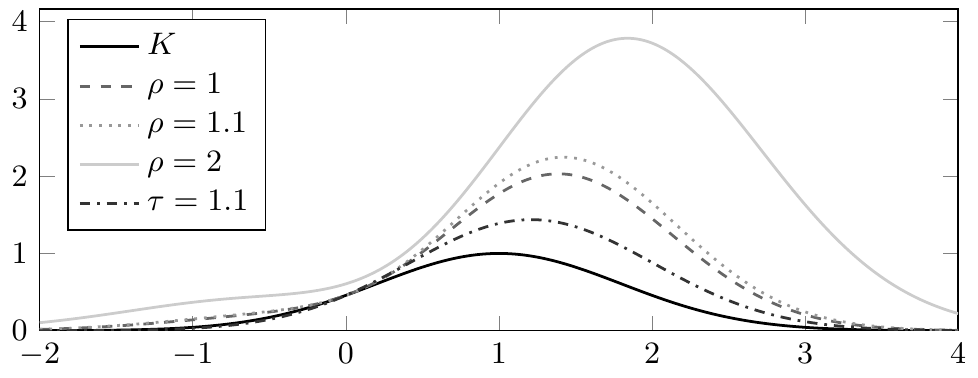}

  \caption{The translates at $t' = 1$ of the Gaussian kernel $K$ in~\eqref{eq:gauss-kernel}, the kernels $K_{A_\rho,\Phi}$ in~\eqref{eq:gauss-kernel-hyperharmonic} for $\rho \in \{1, 1.1, 2\}$, and the kernel in~\eqref{eq:gauss-kernel-tau} with $\tau=1.1$. The length-scale is $\ell = 0.8$.} \label{fig:gaussian}
  
\end{figure}

\subsection{Power Series Kernels} \label{sec:power-series-kernels}

A \emph{power series kernel}~\citep{Zwicknagl2009,ZwicknaglSchaback2013} is a kernel of the form
\begin{equation} \label{eq:power-series-kernel}
  K(t, t') = \sum_{n=0}^\infty \frac{w_n}{(n!)^2} (tt')^n
\end{equation}
for positive coefficients $w_n$ such that the series converges for all $t,t' \in T$.
These kernels do not appear have seen much use in the probabilistic setting, but are useful in functional analysis and approximation theory because, for particular choices of the coefficients, they reproduce important classical function spaces.
For instance, the selection $w_n = (n!)^2$ yields the \emph{Szeg\H{o} kernel}
\begin{equation} \label{eq:szego-kernel}
  K(t,t') = \sum_{n=0}^\infty (tt')^n = \frac{1}{1-tt'}
\end{equation}
which converges as long as $T \subset (-1,1)$ and whose complex extension is the reproducing kernel of the Hardy space~$\mathbb{H}^2_1$ on the unit disc.
The \emph{exponential kernel}
\begin{equation} \label{eq:exponential-kernel}
  K(t,t') = \sum_{n=0}^\infty \frac{1}{n!} (tt')^n = \exp(tt'),
\end{equation}
obtained by setting $w_n = n!$, is power series kernel that is convergent on the whole real line.

From the expansion in~\eqref{eq:power-series-kernel} we see that the functions
$\phi_n(t) = t^n \sqrt{w_n}/ n!$ for $n \geq 0$
form an orthonormal basis of $H(K)$ and therefore a scaled kernel is
\begin{equation*}
  K_{A,\Phi}(t,t') = \sum_{n=0}^\infty \frac{\alpha_n w_n}{(n!)^2} (tt')^n,
\end{equation*}
which itself is a power series kernel.
By Theorem~\ref{thm:driscoll-scaled} the samples from $(X(t))_{t \in T} \sim \GP(0, K)$ are in the RKHS of a different power series kernel
\begin{equation*}
  \bar{K}(t,t') = \sum_{n=0}^\infty \frac{\bar{w}_n}{(n!)^2} (tt')^n \quad \text{ if and only if } \quad \sum_{n=0}^\infty \frac{w_n}{\bar{w}_n} < \infty.
\end{equation*}
For example, samples from a Gaussian process with the exponential covariance kernel~\eqref{eq:exponential-kernel} are in the RKHS of the Szeg\H{o} kernel~\eqref{eq:szego-kernel} if $T \subset (-1,1)$ because
$\sum_{n=0}^\infty w_n / \bar{w}_n = \sum_{n=0}^\infty n! / (n!)^2 < \infty$.
Its samples are \emph{not} in the RKHS of the kernel
\begin{equation*}
  \bar{K}(t,t') = 1 + \sum_{n=1}^\infty \frac{n! n}{(n!)^2} (tt')^n = 1 + tt' \exp(tt'),
\end{equation*}
where $w_0 = 1$ and $w_n = n!n$ for $n \geq 1$, because
$\sum_{n=0}^\infty w_n / \bar{w}_n = 1 + \sum_{n=1}^\infty n^{-1} = \infty$.
Note that this is essentially the same example as the one for the Gaussian kernel that involved the kernel~\eqref{eq:gauss-kernel-harmonic}.

\section{Application to Maximum Likelihood Estimation} \label{sec:MLE}

Gaussian processes are often used to model deterministic data-generating functions in, for example, design of computer experiments~\citep{Sacks1989} and probabilistic numerical computation~\citep{Cockayne2019, Diaconis1988}.
In applications it is typical that the covariance kernel has parameters, such as the length-scale parameter $\ell > 0$ in~\eqref{eq:gauss-kernel}, which are estimated from the data.
Maximum likelihood estimation is one of the most popular approaches to estimate the kernel parameters; see, for example,~\citep[Section~5.4.1]{RasmussenWilliams2006} or~\citep[Chapter~3]{Santner2003}.
This section uses Theorem~\ref{thm:driscoll-scaled} to explain the behaviour of the maximum likelihood estimate of the kernel scaling parameter that Xu and Stein~\citep{XuStein2017} observed recently.

Let $K \colon T \times T \to \R$ be a positive-definite kernel and $K_\sigma = \sigma^2 K$ a kernel parametrised by a non-negative scaling parameter $\sigma$.
Suppose that the data consist of evaluations of a function $f \colon T \to \R$ at distinct points $t_1, \ldots, t_N \in T$ and that $f$ is modelled as a Gaussian process $(X_\sigma(t))_{t \in T} \sim \GP(0, K_\sigma)$.
It is easy to compute that the maximum likelihood estimate $\hat{\sigma}_{f,N}$ of $\sigma$ is (see the references above)
\begin{equation*}
  \hat{\sigma}_{f,N} = \sqrt{ \frac{\mathsf{f}_N^\mathsf{T}\mathsf{K}_N^{-1} \mathsf{f}_N}{N} },
\end{equation*}
where $\mathsf{f}_N = (f(t_1), \ldots, f(t_N)) \in \R^N$ is a column vector and $\mathsf{K}_N \in \R^{N \times N}$ is the positive-definite covariance matrix with elements $(\mathsf{K}_N)_{ij} = K(t_i, t_j)$.
Suppose for a moment that $f$ is a zero-mean Gaussian process with covariance $\sigma_0^2 K$.
Then
\begin{equation*}
  \mathbb{E}_f \big[ \hat{\sigma}_{f,N}^2 \big] = \frac{\mathbb{E}_f \big[ \mathsf{f}_N^\mathsf{T}\mathsf{K}_N^{-1} \mathsf{f}_N \big] }{N} = \frac{\mathbb{E}_f \big[ \mathrm{tr} ( \mathsf{f}_N \mathsf{f}_N^\mathsf{T}\mathsf{K}_N^{-1} ) \big] }{N} = \frac{\mathrm{tr}( \sigma_0^2 \mathsf{K}_N \mathsf{K}_N^{-1} ) }{N} = \sigma_0^2.
\end{equation*}
This suggests that $\hat{\sigma}_{f,N}$ ought to tend to a constant as $N \to \infty$ if $f$ is a \emph{deterministic} function which is ``akin'' to the samples of $(X(t))_{t \in T} \sim \GP(0, K)$.
Theorem~\ref{thm:sample-set-all} can be interpreted as saying that functions of the form $f = \sum_{n=1}^\infty f_n \phi_n$ for coefficients $f_n$ which are bounded away from zero and infinity are akin to the samples.
For example, the function $f(t) = \exp(- t^2/(2\ell^2)) \sum_{n=0}^\infty t^n / ( \ell^n \sqrt{n!})$
can be thought of as a sample from a Gaussian process with the Gaussian covariance kernel in~\eqref{eq:gauss-kernel}.

When $H(K)$ is norm-equivalent to the Sobolev space $W_2^{s + d/2}(T)$ on a suitable bounded domain $T \subset \R^d$, \citet{Karvonen-MLE2020} have used results from scattered data approximation literature to essentially argue that the maximum likelihood estimate decays to zero if $f$ is too regular to be a sample from $X$ and explodes if $f$ is too irregular in the sense that, assuming the points $t_i$ cover $T$ sufficiently uniformly,
\begin{equation*}
  \lim_{N \to \infty} \hat{\sigma}_{f, N} = 0 \: \text{ if } \: f \in W_2^s(T) \quad \text{ and } \quad \lim_{N \to \infty} \hat{\sigma}_{f, N} = \infty \: \text{ if } \: f \in W_2^r(T) \setminus W_2^{s - \varepsilon}(T)
\end{equation*}
for any $\varepsilon > 0$ and $r \in (d/2, s - \varepsilon)$.
Recall from Section~\ref{sec:related} that the samples are contained in $W_2^r(T)$ if and only if $r < s$.
This suggests the conjecture that, for any kernels $K_0$, $K_1$, and $K_2$ such that $\mathbb{P}[X \in H(K_0)] = 0$, $\mathbb{P}[X \in H(K_1)] = \mathbb{P}[X \in H(K_2)] = 1$, and $H(K_2) \subset H(K_1)$, it should hold that
\begin{equation} \label{eq:mle-conjecture}
  \lim_{N \to \infty} \hat{\sigma}_{f, N} = 0 \: \text{ if } \: f \in H(K_0) \quad \text{ and } \quad \lim_{N \to \infty} \hat{\sigma}_{f, N} = \infty \: \text{ if } \: f \in H(K_1) \setminus H(K_2)
\end{equation}
under some assumptions on $T$ and the points $t_i$.
We now show that the behaviour of $\hat{\sigma}_{f,N}$ observed and conjectured by \citet{XuStein2017} when $K$ is the Gaussian kernel in~\eqref{eq:gauss-kernel} and $f$ is a monomial agrees with this conjecture if the kernels $K_0$, $K_1$, and $K_2$ are obtained via hyperharmonic scalings.

Let $K$ be the Gaussian kernel in~\eqref{eq:gauss-kernel}.
Let $T = [0,1]$ and consider the uniform points $t_i = i/N$ for $i=1,\ldots,N$.
\citet{XuStein2017} used explicit Cholesky decomposition formulae derived in~\citep{LohKam2000} to prove that, for certain positive constants $C_{\ell,0}$ and $C_{\ell,1}$,
\begin{equation*}
   \hat{\sigma}_{f,N}^2 \sim C_{\ell,0} N^{-1/2} \:\: \text{ if } \:\: f \equiv 1 \quad \text{ and } \quad \liminf_{N \to \infty} N^{-1/2} \hat{\sigma}_{f,N}^2 \geq C_{\ell,1} \:\: \text{ if } \:\: f(t) = t
\end{equation*}
as $N \to \infty$.
Furthermore, they conjectured that
\begin{equation} \label{eq:XuStein-conjecture}
  \hat{\sigma}_{f,N}^2 \sim C_{\ell,p} N^{p-1/2} \quad \text{ as } \quad N \to \infty
\end{equation}
if $f(t) = t^p$ for any $p \geq 0$ and a positive constant $C_{\ell,p}$.\footnote{As pointed out by \citet[Section~4.1]{DetteZhigljavsky2021}, it appears that the constant $C_{\ell,p} = \ell^{2p}/(\sqrt{2\pi}(p+1/2))$ in~\citep{XuStein2017} is erroneous. By truncating the expansion~\eqref{eq:monomial-as-gauss} after $N$ terms one arrives, after some relatively straightforward computations, to the conjecture that the constant is $C_{\ell,p} = 2^p \ell^{2p}/(\sqrt{\pi}(p+1/2))$.}
In particular, the conjecture~\eqref{eq:XuStein-conjecture} implies that $\hat{\sigma}_{f,N} \to 0$ if $f$ is a constant function and $\hat{\sigma}_{f,N} \to \infty$ if $f(t) = t^p$ for $p \in \N$.
\citet[p.\@~142]{XuStein2017} lacked an intuitive explanation for these findings.
We employ a technique similar to that used in~\citep{Minh2010} to demonstrate that this phenomenon can be explained by the conjecture in~\eqref{eq:mle-conjecture}.

Any monomial of degree $p \in \N_0$ can be written in terms of the orthonormal basis~\eqref{eq:gauss-basis} as
\begin{equation} \label{eq:monomial-as-gauss}
  t^p = \exp\bigg(\frac{t^2}{2\ell^2} \bigg) t^p \exp\bigg(\!-\frac{t^2}{2\ell^2} \bigg) = \sum_{n=0}^\infty \frac{t^{2n+p}}{2^n \ell^{2n} n!} \exp\bigg(\!-\frac{t^2}{2\ell^2} \bigg) = \ell^p \sum_{n=0}^\infty \frac{\sqrt{(2n+p)!}}{2^n n!} \phi_{2n+p}(t).
\end{equation}
Let $A = (\alpha_n)_{n=1}^\infty$ be a $\Phi$-scaling of $H(K)$. If the $p$th monomial is to be an element of the scaled RKHS $H(K_{A,\Phi})$ it must, by Proposition~\ref{thm:scaled-RKHS}, hold that
\begin{equation} \label{eq:MLE-scaled-condition}
  \sum_{n=0}^\infty \alpha_{2n+p}^{-1} \frac{(2n+p)!}{2^{2n} (n!)^2} < \infty.
\end{equation}
From Stirling's formula and $\lim_{n \to \infty} (1+\frac{p}{2n})^{2n} = \e^p$ we get
\begin{align*} 
    \frac{(2n+p)!}{2^{2n} (n!)^2} \sim \frac{\sqrt{2\pi} (2n+p)^{2n+p+1/2} \e^{-(2n+p)}}{2\pi 2^{2n} n^{2n+1} \e^{-2n}} &= \frac{1}{\sqrt{2\pi} \e^p} \frac{(2n+p)^{2n+p+1/2}}{(2n)^{2n}n} \\
    &= \frac{1}{\sqrt{2\pi} \e^p} \bigg( 1 + \frac{p}{2n} \bigg)^{2n} \frac{(2n+p)^{p+1/2}}{n} \\
    &\sim \frac{2^p}{\sqrt{\pi}} \, n^{p-1/2}
\end{align*}
as $n \to \infty$.
Therefore~\eqref{eq:MLE-scaled-condition} holds if and only if $\sum_{n=1}^\infty \alpha_{2n+p}^{-1} n^{p-1/2} < \infty$.
If we constraint ourselves to hyperharmonic scalings this is clearly equivalent to $\alpha_n = n^\rho$ for $\rho > p+1/2$.
But since $\sum_{n=1}^\infty n^{-\rho} < \infty$ if and only if $\rho > 1$, it follows that (a) if $p < 1/2$, then $f(t) = t^p$ is an element of a hyperharmonic scaled RKHS $H(K_{A,\Phi})$ such that $\mathbb{P}[ X \in H(K_{A,\Phi})] = 0$ and (b) if $p \geq 1/2$, then $f(t) = t^p$ is an element of a hyperharmonic scaled RKHS $H(K_{A,\Phi})$ if and only if $\mathbb{P}[ X \in H(K_{A,\Phi})] = 1$.
These observations and the conjecture in~\eqref{eq:XuStein-conjecture} are compatible with our conjecture in~\eqref{eq:mle-conjecture}:
\begin{itemize}
\item If $f$ is a constant function, then $\hat{\sigma}_{f,N}$ tends to zero as $N \to \infty$ and, for $A = (n^\rho)_{n=0}^\infty$ and any $\rho \in (1/2, 1]$, $f$ is an element of the RKHS $H(K_0) = H(K_{A,\Phi})$ which is of zero measure.
\item If $f(t) = t^p$ for $p \in \N$, then $\hat{\sigma}_{f,N}$ tends to infinity (or is conjectured to) as $N \to \infty$.
  Let $H(K_1) = H(K_{A,\Phi})$ and $H(K_2) = H(K_{A',\Phi})$ for hyperharmonic scalings $A = (n^\rho)_{n=0}^\infty$ and $A' = (n^{\sigma})_{n=0}^\infty$, where \sloppy{${\rho > p + 1/2}$} and \sloppy{${\sigma \in (1, p + 1/2]}$}.
  Then both $H(K_1)$ and $H(K_2)$ are of measure one and $f \in H(K_1) \setminus H(K_2)$.
\end{itemize}
In other words, the constant function is too close to $H(K)$, or too ``regular'', to be a sample while higher-order monomials are too far from $H(K)$, or too ``irregular''.

\section*{Acknowledgements}

This work was supported by the Lloyd's Register Foundation Programme for Data-Centric Engineering at the Alan Turing Institute, United Kingdom and the Academy of Finland postdoctoral researcher grant \#338567 ``Scalable, adaptive and reliable probabilistic integration''.
I am grateful to Motonobu Kanagawa for his patience in enduring persistent questions over the course of a conference in Sydney in July~2019 during which the main ideas underlying this article were first conceived and to Chris Oates for extensive commentary on the draft at various stages of preparation.
Discussions with Anatoly Zhigljavsky led to improvements in Section~\ref{sec:MLE}.

\providecommand{\BIBYu}{Yu}


\begin{thebibliography}{}

\bibitem[Adler, 1990]{Adler1990}
Adler, R.~J. (1990).
\newblock {\em An Introduction to Continuity, Extrema, and Related Topics for
  General {G}aussian Processes}.
%% \newblock Number~12 in Lecture Notes--Monograph Series. Institute of
%%   Mathematical Statistics.
\newblock Institute of Mathematical Statistics.

\bibitem[Bogachev, 1998]{Bogachev1998}
Bogachev, V.~I. (1998).
\newblock {\em {G}aussian Measures}.
%% \newblock Number~62 in Mathematical Surveys and Monographs. American
%%   Mathematical Society.
\newblock American Mathematical Society.

\bibitem[Bull, 2011]{Bull2011}
Bull, A.~D. (2011).
\newblock Convergence rates of efficient global optimization algorithms.
\newblock {\em J. Mach. Learn. Res.}, 12:2879--2904.

\bibitem[Cavoretto et~al., 2015]{Cavoretto2015}
Cavoretto, R., Fasshauer, G.~E., and McCourt, M. (2015).
\newblock An introduction to the {H}ilbert-{S}chmidt {SVD} using iterated
  {B}rownian bridge kernels.
\newblock {\em Numer. Algorithms}, 68:393--422.

\bibitem[Cialenco et~al., 2012]{Cialenco2012}
Cialenco, I., Fasshauer, G.~E., and Ye, Q. (2012).
\newblock Approximation of stochastic partial differential equations by a
  kernel-based collocation method.
\newblock {\em Int. J. Comput. Math.},
  89(18):2543--2561.

\bibitem[Cockayne et~al., 2019]{Cockayne2019}
Cockayne, J., Oates, C., Sullivan, T., and Girolami, M. (2019).
\newblock {B}ayesian probabilistic numerical methods.
\newblock {\em SIAM Rev.}, 61(4):756--789.

\bibitem[Dette and Zhigljavsky, 2021]{DetteZhigljavsky2021}
Dette, H. and Zhigljavsky, A. A. (2021).
\newblock Reproducing kernel {H}ilbert spaces, polynomials and the classical
  moment problems.
\newblock {\em SIAM/ASA J. Uncertainty Quantif.}, 9(4):1589--1614.

\bibitem[Diaconis, 1988]{Diaconis1988}
Diaconis, P. (1988).
\newblock Bayesian numerical analysis.
\newblock In {\em Statistical decision theory and related topics IV}, volume~1,
  pages 163--175. Springer-Verlag New York.

\bibitem[Driscoll, 1973]{Driscoll1973}
Driscoll, M.~F. (1973).
\newblock The reproducing kernel {H}ilbert space structure of the sample paths
  of a {G}aussian process.
\newblock {\em Z. Wahrscheinlichkeit}, 26(4):309--316.

\bibitem[Engl et~al., 1996]{EnglHankeNeubauer1996}
Engl, H.~W., Hanke, M., and Neubauer, A. (1996).
\newblock {\em Regularization of Inverse Problems}.
\newblock Springer.

\bibitem[Flaxman et~al., 2016]{Flaxman2016}
Flaxman, S., Sejdinovic, D., Cunningham, J., and Filippi, S. (2016).
\newblock {B}ayesian learning of kernel embeddings.
\newblock In {\em Uncertainty in Artificial Intelligence}, pages 182--191.

\bibitem[Fortet, 1973]{Fortet1973}
Fortet, R. (1973).
\newblock Espaces à noyau reproduisant et lois de probabilités des fonctions
  aléatoires.
\newblock {\em Ann. I. H. Poincare B},
  9(1):41--58.

\bibitem[Gualtierotti, 2015]{Gualtierotti2015}
Gualtierotti, A.~F. (2015).
\newblock {\em Detection of Random Signals in Dependent {G}aussian Noise}.
\newblock Springer.

\bibitem[Irrgeher and Leobacher, 2015]{IrrgeherLeobacher2015}
Irrgeher, C. and Leobacher, G. (2015).
\newblock High-dimensional integration on $\mathbb{R}^d$, weighted {H}ermite
  spaces, and orthogonal transforms.
\newblock {\em J. Complexity}, 31(2):174--205.

\bibitem[Kallianpur, 1970]{Kallianpur1970}
Kallianpur, G. (1970).
\newblock Zero-one laws for {G}aussian processes.
\newblock {\em T. Am. Math. Soc.},
  149:199--211.

\bibitem[Kallianpur, 1971]{Kallianpur1971}
Kallianpur, G. (1971).
\newblock Abstract {W}iener processes and their reproducing kernel {H}ilbert
  spaces.
\newblock {\em Z. Wahrscheinlichkeit}, 17:113--123.

\bibitem[Kanagawa et~al., 2018]{Kanagawa2018}
Kanagawa, M., Hennig, P., Sejdinovic, D., and Sriperumbudur, B.~K. (2018).
\newblock Gaussian processes and kernel methods: A review on connections and
  equivalences.
\newblock {\em arXiv:1807.02582v1}.

\bibitem[Karvonen et~al., 2020]{Karvonen-MLE2020}
Karvonen, T., Wynne, G., Tronarp, F., Oates, C.~J., and Särkkä, S. (2020).
\newblock Maximum likelihood estimation and uncertainty quantification for
  {G}aussian process approximation of deterministic functions.
\newblock {\em SIAM/ASA J. Uncertainty Quantif.}, 8(3):926--958.

\bibitem[Knopp, 1951]{Knopp1951}
Knopp, K. (1951).
\newblock {\em Theory and Application of Infinite Series}.
\newblock Blackie \& Son, 2nd edition.

\bibitem[Krein and Petunin, 1966]{KreinPetunin1966}
Krein, S.~G. and Petunin, {\BIBYu}.~I. (1966).
\newblock Scales of {B}anach spaces.
\newblock {\em Russ. Math. Surv.}, 21(2):85--159.

\bibitem[LePage, 1973]{LePage1973}
LePage, R. (1973).
\newblock Subgroups of paths and reproducing kernels.
\newblock {\em Ann. Probab.}, 1(2):345--347.

\bibitem[Loh and Kam, 2000]{LohKam2000}
Loh, W.-L. and Kam, T.-K. (2000).
\newblock Estimating structured correlation matrices in smooth {G}aussian
  random field models.
\newblock {\em Ann. Stat.}, 28(3):880--904.

\bibitem[Luki{\'c}, 2004]{Lukic2004}
Luki{\'c}, M.~N. (2004).
\newblock Integrated {G}aussian processes and their reproducing kernel
  {H}ilbert spaces.
\newblock In {\em Stochastic Processes and Functional Analysis: A Volume of
  Recent Advances in Honor of {M}.\ {M}.\ {R}ao}, pages 241--263. Marcel Dekker.

  
\bibitem[Luki{\'c} and Beder, 2001]{LukicBeder2001}
Luki{\'c}, M.~N. and Beder, J.~H. (2001).
\newblock Stochastic processes with sample paths in reproducing kernel
  {H}ilbert spaces.
\newblock {\em T. Am. Math. Soc.},
  353(10):3945--3969.

\bibitem[Minh, 2010]{Minh2010}
Minh, H.~Q. (2010).
\newblock Some properties of {G}aussian reproducing kernel {H}ilbert spaces and
  their implications for function approximation and learning theory.
\newblock {\em Constr. Approx.}, 32(2):307--338.

\bibitem[M{\"o}rters and Peres, 2010]{MortersPeres2010}
M{\"o}rters, P. and Peres, Y. (2010).
\newblock {\em {B}rownian Motion}.
\newblock Cambridge University Press.

\bibitem[Parzen, 1959]{Parzen1959}
Parzen, E. (1959).
\newblock Statistical inference on time series by {H}ilbert space methods, {I}.
\newblock Technical Report~23, Applied Mathematics and Statistics Laboratory,
  Stanford University.

\bibitem[Parzen, 1963]{Parzen1963}
Parzen, E. (1963).
\newblock Probability density functionals and reproducing kernel {H}ilbert
  spaces.
\newblock In {\em Proceedings of the Symposium on Time Series Analysis}, pages
  155--169. John Wiley \& Sons.

\bibitem[Paulsen and Raghupathi, 2016]{Paulsen2016}
Paulsen, V.~I. and Raghupathi, M. (2016).
\newblock {\em An Introduction to the Theory of Reproducing Kernel {H}ilbert
  Spaces}.
\newblock Cambridge University Press.

\bibitem[Rakotomamonjy and Canu, 2005]{RakotomamonjyCanu2005}
Rakotomamonjy, A. and Canu, S. (2005).
\newblock Frames, reproducing kernels, regularization and learning.
\newblock {\em J. Mach. Learn. Res.}, 6:1485--1515.

\bibitem[Rasmussen and Williams, 2006]{RasmussenWilliams2006}
Rasmussen, C.~E. and Williams, C. K.~I. (2006).
\newblock {\em Gaussian Processes for Machine Learning}.
\newblock MIT Press.

\bibitem[Sacks et~al., 1989]{Sacks1989}
Sacks, J., Welch, W.~J., M. T.~J., and Wynn, H.~P. (1989).
\newblock Design and analysis of computer experiments.
\newblock {\em Stat. Sci.}, 4(4):409--435.

\bibitem[Santner et~al., 2003]{Santner2003}
Santner, T.~J., Williams, B.~J., and Notz, W.~I. (2003).
\newblock {\em The Design and Analysis of Computer Experiments}.
\newblock Springer.

\bibitem[Scheuerer, 2010]{Scheuerer2010}
Scheuerer, M. (2010).
\newblock Regularity of the sample paths of a general second order random
  field.
\newblock {\em Stoch. Proc. Appl.},
  120(10):1879--1897.

\bibitem[Stein, 1999]{Stein1999}
Stein, M.~L. (1999).
\newblock {\em Interpolation of Spatial Data: {S}ome Theory for Kriging}.
\newblock Springer.

\bibitem[Steinwart, 2019]{Steinwart2019}
Steinwart, I. (2019).
\newblock Convergence types and rates in generic {K}arhunen-{L}oève expansions
  with applications to sample path properties.
\newblock {\em Potential Anal.}, 51:361--395.

\bibitem[Steinwart and Christmann, 2008]{Steinwart2008}
Steinwart, I. and Christmann, A. (2008).
\newblock {\em Support Vector Machines}.
Springer.

\bibitem[Steinwart et~al., 2006]{Steinwart2006}
Steinwart, I., Hush, D., and Scovel, C. (2006).
\newblock An explicit description of the reproducing kernel {H}ilbert spaces of
  {G}aussian {RBF} kernels.
\newblock {\em IEEE T. Inform. Theory}, 52(10):4635--4643.

\bibitem[Steinwart and Scovel, 2012]{SteinwartScovel2012}
Steinwart, I. and Scovel, C. (2012).
\newblock {M}ercer's theorem on general domains: {O}n the interaction between
  measures, kernels, and {RKHS}s.
\newblock {\em Constr. Approx.}, 35:363--417.

\bibitem[van~der Vaart and van Zanten, 2011]{VaarZanten2011}
van~der Vaart, A. and van Zanten, H. (2011).
\newblock Information rates of nonparametric {G}aussian process methods.
\newblock {\em J. Mach. Learn. Res.}, 12(6):2095--2119.

\bibitem[Xu and Stein, 2017]{XuStein2017}
Xu, W. and Stein, M.~L. (2017).
\newblock Maximum likelihood estimation for a smooth {G}aussian random field
  model.
\newblock {\em SIAM/ASA J. Uncertainty Quantif.}, 5(1):138--175.

\bibitem[Xu and Zhang, 2009]{XuZhang2009}
Xu, Y. and Zhang, H. (2009).
\newblock Refinement of reproducing kernels.
\newblock {\em J. Mach. Learn. Res.}, 10:107--140.

\bibitem[Zhang and Zhao, 2013]{ZhangZhao2013}
Zhang, H. and Zhao, L. (2013).
\newblock On the inclusion relation of reproducing kernel {H}ilbert spaces.
\newblock {\em Anal. Appl.}, 11(2):1350014.

\bibitem[Zwicknagl, 2009]{Zwicknagl2009}
Zwicknagl, B. (2009).
\newblock Power series kernels.
\newblock {\em Constr. Approx.}, 29(1):61--84.

\bibitem[Zwicknagl and Schaback, 2013]{ZwicknaglSchaback2013}
Zwicknagl, B. and Schaback, R. (2013).
\newblock Interpolation and approximation in {T}aylor spaces.
\newblock {\em J. Approx. Theory}, 171:65--83.

\end{thebibliography}
\end{document}